\documentclass[english,final]{article}
\usepackage[T1]{fontenc}
\usepackage[latin9]{inputenc}
\usepackage{babel}
\usepackage{amsmath}
\usepackage{amsthm}
\usepackage{amssymb}
\usepackage{stmaryrd}
\usepackage{graphicx}
\usepackage{wasysym}
\usepackage{color}
\usepackage{multirow}
\usepackage[hmargin={28mm,28mm},vmargin={30mm,35mm}]{geometry}
\usepackage[bold]{hhtensor}
\usepackage[unicode=true,pdfusetitle,
bookmarks=true,bookmarksnumbered=false,bookmarksopen=false,
breaklinks=false,pdfborder={0 0 1},backref=false,colorlinks=false]
{hyperref}

\usepackage[color,notref,notcite]{showkeys}
\definecolor{labelkey}{rgb}{0.6,0,1}

\newcommand{\figsFolfer}{./figs/}

\graphicspath{{\figsFolfer}}

\theoremstyle{theorem}
\newtheorem{theorem}{Theorem}
\newtheorem{lemma}[theorem]{Lemma}
\newtheorem{proposition}[theorem]{Proposition}

\theoremstyle{remark}
\newtheorem{remark}[theorem]{Remark}
\theoremstyle{definition}
\newtheorem{assumption}[theorem]{Assumption}

\def\genexp{*}
\newcommand{\R}{\mathbb{R}}

\newcommand{\norm}[2][]{\|#2\|_{#1}}
\newcommand{\seminorm}[2][]{|#2|_{#1}}
\def\mesh{\mathfrak T}
\def\F{\mathfrak S}
\def\E{\mathfrak E}
\def\Fi{\F_{\rm int}}
\def\FG{\F_{\Gamma}}
\def\FD{\F_{\rm Dir}}
\def\EG{\mathfrak E_{\Gamma}}
\def\EGi{\mathfrak E_{\Gamma\!,{\rm int}}}

\def\DIVG{\nabla_{\!\Gamma}\!\boldsymbol{\cdot}\!}
\newcommand{\brak}[2]{\left[ #1 \right]_{#2}}
\def\Fl{\mathcal{F}}
\def\FlG{\Fl^\Gamma}
\newcommand{\flux}[2][\Omega]{\mathcal F^{#1}_{#2}}
\def\Card#1{{\rm Card}(#1)}

\def\term{\mathcal T}

\newcommand{\Ccons}[1]{C_{{\rm cons}}^{#1}}

\begin{document}
	
	\title{Design and analysis of finite volume methods for elliptic equations with oblique derivatives; application to Earth gravity field modelling}
	
	\author{J\'er\^ome Droniou\thanks{School of Mathematics, Monash University, Melbourne (Australia), \texttt{jerome.droniou@monash.edu}}
		\and
		Matej Medla \thanks{Department of Mathematics, Faculty of Civil Engineering, Slovak University of Technology,
			Radlinskeho 11, 810 05 Bratislava, Slovak Republic, \texttt{medla@math.sk}}
		\and
		Karol Mikula  \thanks{Department of Mathematics, Faculty of Civil Engineering, Slovak University of Technology,
			Radlinskeho 11, 810 05 Bratislava, Slovak Republic;
			Algoritmy:SK s.r.o., Sulekova 6, 81106 Bratislava, Slovak Republic, \texttt{mikula@math.sk}}
	}
	
	\maketitle

\begin{abstract}
We develop and analyse finite volume methods for the Poisson problem with boundary conditions involving oblique derivatives. We design a generic framework, for finite volume discretisations of such models, in which internal fluxes are not assumed to have a specific form, but only to satisfy some (usual) coercivity and consistency properties. The oblique boundary conditions are split into a normal component, which directly appears in the flux balance on control volumes touching the domain boundary, and a tangential component which is managed as an advection term on the boundary. This advection term is discretised using a finite volume method based on a centred discretisation (to ensure optimal rates of convergence) and stabilised using a vanishing boundary viscosity. A convergence analysis, based on the 3rd Strang Lemma \cite{DPD18}, is conducted in this generic finite volume framework, and yields the expected $\mathcal O(h)$ optimal convergence rate in discrete energy norm.

We then describe a specific choice of numerical fluxes, based on a generalised hexahedral meshing of the computational domain. These fluxes are a corrected version of fluxes originally introduced in \cite{Medla.et.al2018}. We identify mesh regularity parameters that ensure that these fluxes satisfy the required coercivity and consistency properties. The theoretical rates of convergence are illustrated by an extensive set of 3D numerical tests, including some conducted with two variants of the proposed scheme. A test involving real-world data measuring the disturbing potential in Earth gravity modelling over Slovakia is also presented.

\end{abstract}
	
	\tableofcontents{}
	
	\section{Introduction}
	
	We consider in this work a Laplace equation with oblique boundary conditions:
	\begin{subequations}\label{eq:model}
		\begin{align}
			-\Delta \overline{T}(\vec{x})  =0,&\quad \vec{x}\in\Omega\label{eq:orig_PDE}\\
			\nabla \overline{T}(\vec{x})\cdot\vec{V}(\vec{x})  =g(\vec{x}),&\quad \vec{x}\in\Gamma\label{eq:obliq_BC}\\
			\overline{T}(\vec{x})  =0,&\quad \vec{x}\in\partial\Omega\backslash\Gamma,\label{eq:dir_BC}
		\end{align}
	\end{subequations}
	where $\Omega$ is a bounded domain in $\R^3$ with piecewise $C^2$ boundary, $\Gamma$ is a relatively open subset of $\partial\Omega$ that is fully contained in a  smooth component of this boundary, $g\in L^2(\partial\Omega)$ and $\vec{V}$ is a $C^1$ vector field such that $\vec{V}(\vec{x})\cdot\vec{n}(\vec{x})\not=0$ for all $\vec{x}\in \Gamma$. Here, $\vec{n}$ denotes the outer normal to $\partial\Omega$. We also assume that the $(d-1)$-dimensional measure of $\partial\Omega\backslash \Gamma$ is non-zero.
	On $\Gamma$, $\vec{V}$ can be decomposed into a normal and a tangential component to $\Gamma$. After renormalising $g$ we can assume that the normal component is $\vec{n}$, and thus that
	\begin{equation}\label{eq:defV}
		\vec{V}(\vec{x})=\vec{n}(\vec{x})+\vec{W}(\vec{x}),\quad\forall \vec{x}\in\Gamma.
	\end{equation}
	The properties of $\Gamma$ ensure that $\vec{W}$ is a $C^1$ tangential vector field on $\Gamma$. 
	
	A motivation to study boundary value problem (BVP) \eqref{eq:model} comes from Earth gravity field modelling. The Earth gravity potential $G$ fulfils outside the Earth a non-homogeneous elliptic equation
	\begin{equation}
		\Delta G(\vec{x})=2\omega^2,
		\label{eq:non-homogeneous_elliptic_eq}
	\end{equation}
	where $\omega$ is the spin velocity of the Earth \cite{Moritz}. The magnitude of the total gravity vector $\nabla G$ is  called gravity. If the measured gravity is  prescribed on the Earth surface, i.e.
	\begin{equation}\label{eq:grav_pot_grad_norm}
		|\nabla G(\vec{x})|=\overline{g}(\vec{x}),
	\end{equation}
	then Eq. \eqref{eq:non-homogeneous_elliptic_eq} with BC \eqref{eq:grav_pot_grad_norm} represents the so-called nonlinear geodetic BVP for the actual gravity potential~$G$. The existence, uniqueness and other properties to the solution of this problem, and its variants, were studied extensively in physical geodesy community, see e.g. \cite{Backus1968, Grafarend1971, Koch1972,  Bjerhammar1983, Grafarend1989, Sacerdote1989, Heck1989, Diaz2006,Diaz2011}.
	
	
	In Earth gravity field modelling,  the actual gravity field $G$  is usually expressed as a sum of the selected model field $U$ and the remainder $\overline{T}$, i.e.
	\begin{equation}
		G(\vec{x})=U(\vec{x})+\overline{T}(\vec{x}).
	\end{equation}
	If the model field $U$ is generated by an ellipsoid with the same mass as the Earth, rotating with the same spin velocity $\omega$ and with the constant surface potential equal to the geopotential $W_0$ (see \cite{SanchesW0} for a definition of $W_0$, the potential $G$ on the mean sea surface level), then $U$ is called the normal gravity potential and $\overline{T}$ is called the disturbing potential. This potential $\overline{T}$ has no centrifugal component and it is generally accepted that the disturbing potential satisfies the Laplace equation $\Delta \overline{T}=0$  outside the Earth, see e.g. \cite{Moritz, Holota1997}.
	In the satellite era, people have been able to consider a bounded domain $\Omega$ outside the Earth where an upper part of the boundary is given as a sphere at altitude of a chosen satellite mission, and the bottom part $\Gamma \subset \partial\Omega$ is given by a subset of the Earth surface \cite{Faskova2010, Medla.et.al2018}. On this bottom part  $\Gamma$ the nonlinear BC \eqref{eq:grav_pot_grad_norm} is given and, on the upper part, as well as on the side boundaries if one focuses on a tesseroid above the Earth, the Dirichlet-type BC obtained from satellite gravity missions can be prescribed. This allows us to fix a solution to the satellite data $\overline{T}_{\mbox{\small\textsc{sat}}}$. See Figure \ref{fig:region_sat} for an illustration. 

	\begin{figure}
		\centering
		\input{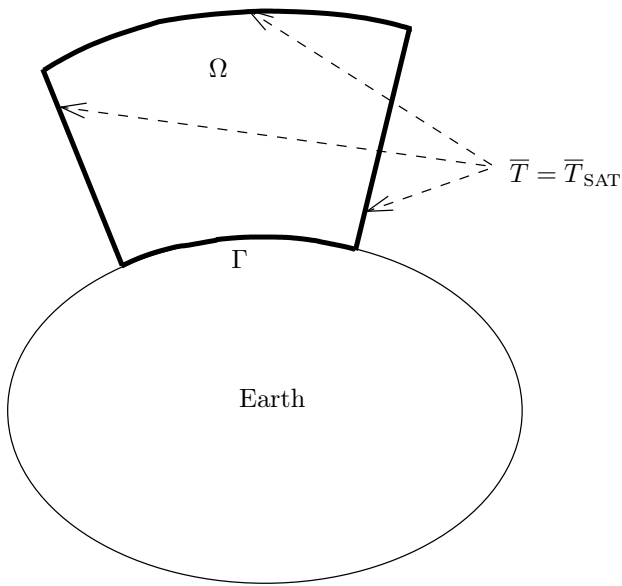}
		\caption{Illustration of the type of domains and boundary conditions that can be considered, using satellite data, in Earth gravity field simulation.}
		\label{fig:region_sat}
	\end{figure}

Such nonlinear satellite-fixed geodetic BVP \cite{Macak2016} for the disturbing potential $\overline{T}$ can be formulated as follows
	\begin{align}
		\Delta \overline{T}(\vec{x})=0, &\quad\vec{x}\in \Omega,\label{eq:non-linear_potential_laplace}\\
		|\nabla (\overline{T}+U)(\vec{x})| = \overline{g}(\vec{x}),& \quad \vec{x}\in  \Gamma,\label{eq:non-linear_potential_oblique_BC}  \\
		\overline{T}(\vec{x}) = \overline{T}_{\mbox{\small\textsc{sat}}}(\vec{x}),& \quad \vec{x}\in \partial\Omega\backslash\Gamma.\label{eq:non-linear_potential_dirichlet_BC}
	\end{align}
	Using the relation $|\vec{\xi}|=\frac{\vec{\xi}}{|\vec{\xi}|}\cdot\vec{\xi}$ for $\vec{\xi}\in\R^3\backslash\{0\}$ we get \eqref{eq:non-linear_potential_oblique_BC} in the form
	\begin{equation*}
		\frac{\nabla (\overline{T}+U)(\vec{x})}{|\nabla (\overline{T}+U)(\vec{x})|}\cdot \nabla (\overline{T}+U)(\vec{x}) = \overline{g}(\vec{x}).
	\end{equation*}
	Letting
	\begin{equation}\label{eq:grav_pot_grad}
		\vec{v}(\vec{x})=\frac{\nabla (\overline{T}+U)(\vec{x})}{|\nabla (\overline{T}+U)(\vec{x})|}
	\end{equation}
	be the actual gravity vector $\nabla G(\vec{x})$ unit direction, we can rewrite the nonlinear boundary condition \eqref{eq:non-linear_potential_oblique_BC} as follows
	\begin{equation}\label{eq:non-linear_potential_oblique_BC_rewrite}
		\vec{v}(\vec{x})\cdot \nabla \overline{T}(\vec{x}) = \overline{g}(\vec{x})-\vec{v}(\vec{x})\cdot \nabla U(\vec{x}), \quad \vec{x}\in  \Gamma. 
	\end{equation}
	Since the unit vector $\vec{v}(\vec{x})$ is unknown and depends on $\overline{T}(\vec{x})$, boundary condition \eqref{eq:non-linear_potential_oblique_BC_rewrite} is still nonlinear. 
	However, if we set $\overline{T}(\vec{x})=0$ in \eqref{eq:grav_pot_grad}, which means that we approximate the unit direction $\vec{v}(\vec{x})$ of the actual gravity vector  by the unit direction of the normal gravity vector equal to $\frac{\nabla U(\vec{x})}{|\nabla U(\vec{x})|}$, we get a linear(ized) boundary condition 
	\begin{equation}\label{eq:oblique_geo_BC}
		\vec{V}(\vec{x})\cdot \nabla \overline{T}(\vec{x}) = \overline{g}(\vec{x})-\overline{\gamma}(\vec{x}), 
	\end{equation}
	where $\overline{\gamma}(\vec{x})=\vec{V}(\vec{x})\cdot\nabla U(\vec{x})= \frac{\nabla U(\vec{x})}{|\nabla U(\vec{x})|}\cdot\nabla U(\vec{x})=|\nabla U(\vec{x})|$ is the so-called normal gravity. Since all quantities depending on $U$ are given analytically, the equation  \eqref{eq:oblique_geo_BC} represents a linear oblique derivative boundary condition. Together with equations \eqref{eq:orig_PDE} and \eqref{eq:dir_BC}, they are called the fixed gravimetric boundary value problem in the geodetic community \cite{Backus1968,Koch1972,Holota1997,Holota2005,Cunderlik2008,Faskova2010,Medla.et.al2018} and give a basis for determining the Earth gravity field when gravity measurements are known on the Earth surface. When we denote $\overline{g}-\overline{\gamma}=g$ and consider the problem on a bounded domain outside the Earth, we end up with the oblique derivative BVP \eqref{eq:model} (here, $\vec{V}$ has unit length and, as previously mentioned, the decomposition \eqref{eq:defV} is obtained after rescaling $g$).
	
	Let us briefly mention some results that can be found in the literature regarding the numerical approximation of second order equations with oblique derivative boundary conditions.
	In \cite{Bradji2006,Bradji2018} authors deal with the finite volume method for the oblique derivative boundary value problem in 2D case. In \cite{Bradji2006} they consider the oblique BC in the form
	\begin{equation}
		\overline{T}_n(\vec{x})+(\alpha \overline{T})_t(\vec{x})=g(\vec{x}), 
	\end{equation}
	where $\alpha$ is a smooth function, $\overline{T}_n$ is a derivative in the normal direction and $(\alpha \overline{T})_t$ is a derivative in tangential direction. They develop a finite volume scheme based on the upwind principle, prove its convergence and obtain an error estimate of order $\sqrt{h}$.  In \cite{Bradji2018}, convergence results are established for the Poisson and a parabolic equation with oblique derivative boundary condition in which $\alpha$ is constant. The convergence results are not only obtained for the approximate finite volume solutions, but also for their discrete gradients. The error estimate of order $\sqrt{h}$ is obtained theoretically, but numerical experiments presented in these works indicate a first order rate of convergence. In \cite{Barrett1985} the authors present and analyse a 2D finite element method for the oblique derivative boundary problem, where the oblique derivative boundary is given by a graph of a real function. Finite volume methods for solving oblique derivative problems in 3D domains were suggested and numerically investigated in \cite{Macak2012,Macak2014,Macak2015, Medla.et.al2018}. These schemes are based either on upwind or central approximation of the oblique derivative. A numerical approximation of a nonlinear problem with eikonal-type boundary condition (\ref{eq:grav_pot_grad_norm}) was presented in \cite{Macak2016}.  This approximation is based on an iterative update of the oblique derivative condition. 
	
	In this paper we introduce and analyse novel numerical scheme for solving 3D oblique derivative boundary value problems for the Laplace equation. For the first time, there is presented a convergence analysis and error estimates for a finite volume scheme solving the oblique derivative problem in 3D. The model comes from the Earth gravity field modelling on real Earth topography but can be used in other applications as well. The presented numerical approach is general and covers various possible discretisations of the Laplace equation inside the domain and it treats in a robust and stable way the oblique derivative boundary condition. We present numerical tests showing convergence properties of the novel scheme and compare them to further alternative numerical treatments of the oblique derivative. We also present a local Earth gravity field modelling for the region of Slovakia where we compare obtained numerical results with GPS-leveling measurements.
	
	The paper is organised as follows. In Section  \ref{sec:scheme} we present a  generic finite volume method  for solving the oblique derivative problem for the Laplace equation and its error estimates. In Section \ref{sec:tests} we specify approximation of inner and surface fluxes and present results of numerical computations. We also give alternative schemes for oblique derivative treatment and discuss their pros and cons. In Section \ref{sec:conclusion} we present concluding remarks. Appendix \ref{appen:proof} contains proof of error estimate for the generic scheme and Appendix \ref{appen:prop.method} contains proof of coercivity and consistency  of the suggested  inner flux approximation.
	
	\section{Generic finite volume scheme} \label{sec:scheme}
	
	We describe here a generic finite volume approximation of \eqref{eq:model}. The discretisation is based on a recasting of the model to transform the oblique derivative into a normal derivative, handled as a Neumann boundary condition, and a boundary advection--reaction term along $\Gamma$. The method is ``generic'' in the sense that we do not impose any specific expression of the numerical fluxes, only broad assumptions that enable the convergence analysis of the method. Our approach and analysis therefore cover many possible choices of Finite Volume methods for discretising the Laplacian in the domain.
	
	\subsection{Mesh, space of unknowns and interpolant}
	
	Let $\mesh$ be a partition of $\Omega$ into ``generalised'' polyhedral finite volumes $p$, the generalisation coming from the fact that the faces of the polyhedra could be curved (especially those lying on $\Gamma$). The mesh size is $h:=\max_{p\in\mesh}{\rm diam}(p)$. We denote by $\F$ the set of faces of the mesh, and by $\Fi$ the faces contained in $\Omega$. The boundary faces are assumed to be compatible with $\Gamma$ in the sense that each face in $\F\backslash\Fi$ totally lies on $\Gamma$, or totally lies on the Dirichlet boundary $\partial\Omega\backslash\Gamma$. We let $\FG$ be the set of faces on $\Gamma$, and $\FD=\F\backslash(\Fi\cup \FG)$ be the set of faces on $\partial\Omega\backslash\Gamma$.
	
	For each cell $p\in\mesh$ we take a point $\vec{x}_p\in p$ and we denote by $\F(p)$ the set of faces of $p$, so that $\partial p=\cup_{\sigma \in \F(p)}\sigma$.
	If $\sigma\in \F(p)$, $\vec{n}_{p,\sigma}$ is the unit outer normal to $p$ on $\sigma$.
	Every face $\sigma$ in $\FG$ is a face of a unique finite volume $p$; the dependency of $p$ on $\sigma$ is not made explicit as there is no risk of confusion in the formulas. We assume that:
	\begin{equation}\label{assum:mesh}
		\begin{aligned}
			\mbox{Each control volume $p\in\mesh$ has at most one face $\sigma$ in $\FG$ and, in that case, $\vec{x}_p\in\sigma$}.
		\end{aligned}
	\end{equation}
	\begin{remark}[Assumption \eqref{assum:mesh}]
		This assumption is not mandatory, and the design and analysis in the following sections could be adapted to meshes not satisfying \eqref{assum:mesh} (see Remark \ref{rem:assum:mesh}); however, the method we consider in Section \ref{sec:tests} naturally satisfies this property, which is why we assume it in our analysis.
	\end{remark}
	
	For every $p\in\mesh$ and $\sigma\in \F(p)\cap \Fi$, we denote by $q_{p}(\sigma)$ the finite volume such that $\sigma=\overline{p}\cap \overline{q_p(\sigma)}$; here too, no risk of confusion arising we simply denote $q$ for $q_p(\sigma)$. We then set $d_{pq}=|\vec{x}_{p}-\vec{x}_{q}|$. A face $\sigma\in\F(p)\cap\FD$ on the Dirichlet boundary of a cell $p$ is sometimes considered as a ``degenerate'' cell, and also denoted by $q$; for such faces, we pick a point $\vec{x}_q\in\sigma$ and define again $d_{pq}=|\vec{x}_{p}-\vec{x}_{q}|$.
	
	For each $\sigma\in \FG$ we take $\vec{x}_{\sigma}\in\sigma$, and we denote by $\E(\sigma)$ the set of edges $e$ of $\sigma$. The set of all such edges is $\EG=\cup_{\sigma\in\FG}\E(\sigma)$, and the edges that lie in the relative interior of $\Gamma$ are gathered in the set $\EGi$. For $e\in \E(\sigma)$, $\vec{n}_{\sigma,e}$ is the unit normal outward to $\sigma$ on $e$ in the tangent space of $\Gamma$.
	
	If $X$ is a control volume $p$, a face $\sigma$ or an edge $e$, $|X|$ denotes the Lebesgue measure of $X$ in the corresponding dimension of $X$ (dimension 3 for a control volume, 2 for a face, 1 for an edge). 
	
	\medskip
	
	Our space of approximation has unknowns in the finite volumes, on the Dirichlet faces (``degenerate cells''), and on each edge on $\Gamma$, with zero values for Dirichlet faces, and for edges that are not in the relative interior of $\Gamma$:
	\[
	V_{h}:=\{\varphi=((\varphi_{p})_{p\in \mesh},(\varphi_{\sigma})_{\sigma\in\FD},(\varphi_{e})_{e\in \EG})\,:\,\varphi_p\in\R\,,\;\varphi_\sigma=0\,,\;\varphi_e\in\R\,,\;\varphi_e=0\mbox{ if $e\not\in\EGi$}\}.
	\]
	\begin{remark} Introducing the zero-valued unknowns is of course not necessary, but will be useful to simplify some expressions.
	\end{remark}
	The norm on $V_h$ is defined by
	\begin{subequations}\label{eq:norm.Vh}
		\begin{equation}\label{eq:norm.Vh.1}
			\norm[V_h]{\varphi} := \left(\seminorm[V_h,\Omega]{\varphi}^2+h_\Gamma\seminorm[V_h,\Gamma]{\varphi}^2\right)^{1/2},
		\end{equation}
		where $h_\Gamma:=\max_{\sigma\in\FG}{\rm diam}(\sigma)$,
		\begin{equation}\label{eq:norm.Vh.Omega}
			\seminorm[V_h,\Omega]{\varphi}:=\left(\sum_{\sigma\in \F\backslash \FG}\frac{|\sigma|}{d_{pq}}(\varphi_{p}-\varphi_{q})^{2}\right)^{1/2},
		\end{equation}
		and
		\begin{equation}\label{eq:norm.Vh.Gamma}
			\seminorm[V_h,\Gamma]{\varphi}:=\left(\sum_{\sigma\in \FG}\sum_{e\in \E(\sigma)}\frac{|e|}{d_{pe}^\bot}(\varphi_{p}-\varphi_{e})^{2}\right)^{1/2}
		\end{equation}
	\end{subequations}
	where $d_{pe}^\bot$ is the orthogonal distance between $\vec{x}_p$ (which belongs to $\sigma$) and $e$.
	Remember that, in \eqref{eq:norm.Vh.Omega}, $p$ and $q$ are the two cells on each side of $\sigma$ if $\sigma\in\Fi$, and $q=\sigma$ if $\sigma\in\FD$ (so that $\varphi_q=\varphi_\sigma=0$ in that case).
	The term $\seminorm[V_h,\Omega]{\varphi}$ can thus be viewed as a discrete $H^1_0$-(semi)norm in $\Omega$ \cite[Eq. (7.7f)]{gdm}, whilst $\seminorm[V_h,\Gamma]{\varphi}$ plays the role of a discrete $H^1_0$-(semi)norm on the surface $\Gamma$. The presence of this boundary semi-norm, and its scaling by $h_\Gamma$, will be justified by the introduction of a small amount of diffusion on that surface to stabilise a centred approximation of an advective term on $\Gamma$ stemming from the oblique boundary condition (see \eqref{eq:scheme:balance}). Notice that, in \eqref{eq:norm.Vh.Gamma}, Assumption \eqref{assum:mesh} was used to identify the unknown on a face $\sigma\in\FG$ with the value $\varphi_p$ corresponding to $p\in\mesh$ such that $\sigma\in\F(p)$.

	The unknowns in the control volumes $p$ are destined to be approximations of the solution at $\vec{x}_p$, whereas those on the boundary edges approximate the average value of the solution on the corresponding edge. This leads to defining the following interpolant $I_h:C(\overline{\Omega})\to V_h$: for $\varphi\in C(\overline{\Omega})$ such that $\varphi=0$ on $\partial\Omega\backslash\Gamma$,
	\begin{equation}
		\begin{aligned}
			&I_{h}\varphi=((\varphi_{p})_{p\in \mesh},(\varphi_{\sigma})_{\sigma\in\FD},(\varphi_{e})_{e\in \EG})\mbox{ with }\\
			&\varphi_{p}=\varphi(\vec{x}_{p})\quad\forall p\in \mesh,\qquad\varphi_\sigma=\varphi(\vec{x}_q)\quad\forall q=\sigma\in\FD\,,\qquad \varphi_e=\frac{1}{|e|}\int_{e}\varphi\quad\forall e\in \EG.
		\end{aligned}
		\label{eq:interpolant}
	\end{equation}
	The boundary condition $\varphi=0$ on $\partial\Omega\backslash\Gamma$ ensures that $\varphi_\sigma=0$ for all $\sigma\in\FD$, and that $\varphi_e=0$ whenever $e\not\in\EGi$.
	
	\subsection{Prolegomena to the scheme}
	
	Integrating \eqref{eq:obliq_BC} over a control volume $p\in\mesh$, using Green's theorem and introducing $\vec{W}$ defined in \eqref{eq:defV}, it holds
	\begin{align*}
		0=\iiint_{p}-\Delta\overline{T}
		={}&-\iint_{\partial p}\nabla\overline{T}\cdot\vec{n}_{p}\\
		={}&-\sum_{\sigma\in \F(p)\backslash \FG}\iint_{\sigma}\nabla\overline{T}\cdot\vec{n}_{p,\sigma}-\sum_{\sigma\in \F(p)\cap\FG}\iint_{\sigma}\nabla\overline{T}\cdot(\vec{n}_{p,\sigma}+\vec{W}-\vec{W}).
	\end{align*}
	Denoting by $\overline{F}_{p,\sigma}(\overline{T})=-\iint_{\sigma}\nabla\overline{T}\cdot\vec{n}_{p,\sigma}\mathrm{d}\vec{x}$
	the exact fluxes, we invoke the boundary condition \eqref{eq:obliq_BC} to write
	\[
	0=\sum_{\sigma\in \F(p)\backslash \FG}\overline{F}_{p,\sigma}(\overline{T})-\sum_{\sigma\in \F(p)\cap\FG}\iint_{\sigma}g-\nabla\overline{T}\cdot\vec{W}.
	\]
	The vector field $\vec{W}$ is tangential to $\Gamma$ and thus only the tangential gradient of $\overline{T}$ is involved in the quantity $\nabla\overline{T}\cdot\vec{W}$. We can therefore write $\nabla\overline{T}\cdot\vec{W}=\DIVG(\overline{T}\vec{W})-\overline{T}\DIVG\vec{W}$, where $\DIVG$ is the divergence operator on the manifold $\Gamma$. This leads, using the divergence theorem on each face $\sigma\in \F(p)\cap\FG$, to
	\begin{align*}
		\sum_{\sigma\in \F(p)\cap\FG}\iint_{\sigma}g={}&
		\sum_{\sigma\in \F(p)\backslash\FG}\overline{F}_{p,\sigma}(\overline{T})
		+\sum_{\sigma\in \F(p)\cap\FG}\iint_{\sigma}\DIVG\left(\overline{T}\vec{W}\right)-\overline{T}\DIVG\vec{W}\\
		{}=&\sum_{\sigma\in \F(p)\backslash\FG}\overline{F}_{p,\sigma}(\overline{T})
		+\sum_{\sigma\in \F(p)\cap\FG}\sum_{e\in \E(\sigma)}\int_{e}\overline{T}\vec{W}\cdot\vec{n}_{\sigma,e}-\sum_{\sigma\in \F(p)\cap\FG}\iint_{\sigma}\overline{T}\DIVG\vec{W}.
	\end{align*}
	Let us denote by $\brak{\,\overline{T}\vec{W}\cdot\vec{n}}{\sigma,e}=\int_{e}\overline{T}\vec{W}\cdot\vec{n}_{\sigma,e}$
	the exact advection fluxes on the boundary, and by $\brak{\,\overline{T}\DIVG\vec{W}}{\sigma}=\iint_{\sigma}\overline{T}\DIVG\vec{W}\mathrm{d}\vec{x}$ the other contribution (akin to a reaction term) to the boundary term. 
	This shows that the solution to \eqref{eq:model} satisfies, for all $p\in \mesh$,
	\begin{equation}
		\sum_{\sigma\in \F(p)\backslash\FG}\overline{F}_{p,\sigma}(\overline{T})
		+\sum_{\sigma\in \F(p)\cap\FG}\sum_{e\in \E(\sigma)}\brak{\,\overline{T}\vec{W}\cdot\vec{n}}{\sigma,e}
		-\sum_{\sigma\in \F(p)\cap\FG}\brak{\,\overline{T}\DIVG\vec{W}}{\sigma} =\sum_{\sigma\in \F(p)\cap\FG}\iint_{\sigma}g.\label{eq:exact_scheme}
	\end{equation}

	\subsection{Scheme}
	
	The scheme for \eqref{eq:model} is obtained discretising \eqref{eq:exact_scheme}. As previously mentioned, we will assume generic properties on the diffusive numerical fluxes. The advective contribution to the boundary terms is discretised using a centred scheme, to which we add a small amount of (boundary) diffusion for stabilisation purposes. As discussed in Remark \ref{rem:centred}, the choice of a centred discretisation seems crucial to prove optimal error estimates.
	
	Based on our choice of unknowns and interpolant \eqref{eq:interpolant}, we make the following approximation, in which $T=((T_p)_{p\in\mesh},(T_\sigma)_{\sigma\in\FD},(T_e)_{e\in\EG})$ is the sought approximation of $\overline{T}$:
	\begin{equation}\label{eq:approx.brak}
		\brak{\,\overline{T}\vec{W}\cdot\vec{n}}{\sigma,e}\approx T_{e}\brak{\vec{W}\cdot\vec{n}}{\sigma,e}\quad\mbox{ and }
		\quad \brak{\,\overline{T}\DIVG\vec{W}}{\sigma}\approx T_{p}\brak{\DIVG\vec{W}}{\sigma},
	\end{equation}
	where $\brak{\vec{W}\cdot\vec{n}}{\sigma,e}=\int_{e}\vec{W}\cdot\vec{n}_{\sigma,e}$ and $\brak{\DIVG\vec{W}}{\sigma}=\int_\sigma\DIVG\vec{W}$. Here, we used Assumption \eqref{assum:mesh} to utilise $T_p$ as approximate value of $\overline{T}$ on $\sigma\in\F(p)\cap\FG$.
	The exact fluxes $\overline{F}_{p,\sigma}(\overline{T})$, for $p\in\mesh$ and $\sigma\in\F(p)\backslash\FG$, are discretised into numerical fluxes $\flux{p,\sigma}(T)$ that satisfy the following conservativity condition: for all $\varphi\in V_h$ and all $\sigma\in\Fi$,
	\begin{equation}\label{eq:cons.inner}
		\flux{p,\sigma}(\varphi)+\flux{q,\sigma}(\varphi)=0.
	\end{equation}
	We also select numerical diffusion fluxes $\flux[\Gamma]{\sigma,e}(T)$ on the boundary, approximations of $-\int_e \nabla_\Gamma \overline{T}\cdot\vec{n}_{\sigma,e}$ for $\sigma\in\FG$ and $e\in\E(\sigma)$.
	
	\medskip
	
	The resulting finite volume scheme has the form: find $T=((T_p)_{p\in\mesh},(T_\sigma)_{\sigma\in\FD},(T_e)_{e\in\EG})\in V_h$ such that:
	\begin{subequations}\label{eq:scheme}
		\begin{equation}
			\begin{aligned}\label{eq:scheme:balance}
				\sum_{\sigma\in \F(p)\backslash\FG}\flux{p,\sigma}(T)+\sum_{\sigma\in \F(p)\cap\FG}{}&\sum_{e\in \E(\sigma)}T_{e}\brak{\vec{W}\cdot\vec{n}}{\sigma,e}
				-\sum_{\sigma\in \F(p)\cap\FG}T_{p}\brak{\DIVG\vec{W}}{\sigma}\\
				&+Rh_\Gamma\sum_{\sigma\in \F(p)\cap\FG}\sum_{e\in \E(\sigma)}\flux[\Gamma]{\sigma,e}(T) =\sum_{\sigma\in \F(p)\cap\FG}\iint_{\sigma}g,\quad \forall p\in\mesh,
			\end{aligned}
		\end{equation}
		where $R\in (0,+\infty)$ will be adjusted later (see Remark \ref{rem:assum:coer}), and
		\begin{equation}\label{eq:scheme:conserv}
			\flux[\Gamma]{\sigma,e}(T)+\flux[\Gamma]{\tau,e}(T)=0,\quad\forall e\in\EGi\mbox{ with $\sigma,\tau\in\FG$ the two faces on each side of $e$}.
		\end{equation}
	\end{subequations}

	\begin{remark}[Conservativity of the fluxes]
		Because they correspond to a cell-centred finite volume method, the inner fluxes $\flux{p,\sigma}$ must satisfy by design the conservativity condition \eqref{eq:cons.inner} on any vector $\varphi\in V_h$. On the contrary, the fluxes $\flux[\Gamma]{\sigma,e}$ correspond to a cell- and edge-centred method and their conservativity is only imposed on the solution to the finite volume scheme (see Equation \eqref{eq:scheme:conserv}). See also \cite[Sections 3.3.1 and 3.3.3]{DPD18} on this topic.
	\end{remark}
	
	\subsection{Error estimate}
	
	The following assumptions are made on the diffusive fluxes.
	
	\begin{assumption}\label{assum:fluxes} The numerical fluxes satisfy:
		\begin{enumerate}
			\item \emph{Coercivity}: there is $\rho_\Omega>0$ and $\rho_\Gamma>0$ such that, for all $\varphi\in V_h$,
			\begin{equation}\label{eq:coer.flux.Omega}
				\sum_{\sigma\in \F\backslash\FG}\flux{p,\sigma}(\varphi)\left(\varphi_{p}-\varphi_{q}\right)\ge \rho_\Omega \seminorm[V_h,\Omega]{\varphi}^{2},
			\end{equation}
			and
			\begin{equation}\label{eq:coer.flux.Gamma}
				\sum_{\sigma\in \FG}\sum_{e\in\E(\sigma)}\flux[\Gamma]{\sigma,e}(\varphi)\left(\varphi_{p}-\varphi_{e}\right)\ge \rho_\Gamma\seminorm[V_h,\Gamma]{\varphi}^{2}.
			\end{equation}
			\item \emph{Consistency}: there exist constants $\Ccons{\Omega}$ and $\Ccons{\Gamma}$ such that, for all $u\in C^2(\overline{\Omega})$ with $u=0$ on $\partial\Omega\backslash\Gamma$,
			\begin{equation}\label{eq:cons.flux.Omega}
				\left|\flux{p,\sigma}(I_h u)+\iint_\sigma \nabla u\cdot\vec{n}_{p,\sigma}\right|\le \Ccons{\Omega} h|\sigma|\norm[C^2(\overline{\Omega})]{u},\quad\forall p\in\mesh\,,\;\forall\sigma\in\F(p)\backslash\FG,
			\end{equation}
			and
			\begin{equation}\label{eq:cons.flux.Gamma}
				\left|\flux[\Gamma]{\sigma,e}(I_h u)+\int_e \nabla u\cdot\vec{n}_{\sigma,e}\right|\le \Ccons{\Gamma} h_\Gamma|e|\norm[C^2(\overline{\Omega})]{u},\quad\forall \sigma\in\FG\,,\;\forall e\in\E(\sigma).
			\end{equation}
		\end{enumerate}
	\end{assumption}
	
	The error estimate will be established under the assumption that the following mesh regularity factor remains bounded  above:
	\begin{equation}\label{def:regh}
		{\rm reg}_\mesh=\max\left\{\frac{{\rm diam}(p)}{d^\bot_{p,\sigma}}\,:\,p\in\mesh\,,\;\sigma\in\F(p)\backslash\FG\right\}+\max\left\{\frac{{\rm diam}(p)}{{\rm diam}(q)}\,:\,p\in\mesh\,,\;\sigma\in\F(p)\cap\Fi\right\},
	\end{equation}
	where $d^\bot_{p,\sigma}$ is the orthogonal distance between $\vec{x}_p$ and $\sigma$.
	
	\begin{remark}[Interpretation of ${\rm reg}_\mesh$]\label{rem:reg.mesh}
		Bounding ${\rm reg}_\mesh$ above imposes that each $\vec{x}_p$ must be ``well within'' its cell $p$, and that two neighbouring cells must have comparable diameters (which does not prevent local refinement, provided that it is done in layers of smoothly refined meshes).
	\end{remark}
	
	Combining \cite[Lemmas B.21 and B.31]{gdm}, we obtain the following discrete trace inequality: there is $C_{\rm tr}>0$ depending only on $\Omega$, $\Gamma$ and an upper bound of ${\rm reg}_\mesh$ such that, for all $\varphi\in V_h$,
	\begin{equation}\label{trace:ineq}
		\sum_{\sigma\in \FG}|\sigma|\varphi_{p}^{2}\le C_{\rm tr}\seminorm[V_h,\Omega]{\varphi}^2.
	\end{equation}
	
	In the rest of the paper, the notation $a\lesssim b$ means that $a\le Cb$ for a constant $C$ that is independent of the quantities in $a$ and $b$, and of the mesh (but that may depend on $\Omega$, $\vec{W}$, $\Gamma$, $\rho_\Omega$, $\rho_\Gamma$, $\Ccons{\Omega}$, $\Ccons{\Gamma}$, $R$ and an upper bound of ${\rm reg}_\mesh$). We can now state our main error estimate.
	
	\begin{theorem}[Error estimate]\label{th:error.estimate}
		Under Assumption \ref{assum:fluxes}, suppose that $\vec{W}$ satisfies
		\begin{subequations}\label{assum:coer}
			\begin{equation}\label{assum:coer.Omega}
				\norm[C(\Gamma)]{(\DIVG\vec{W})^+}< \frac{2\rho_\Omega}{C_{\rm tr}},
			\end{equation}
			where $(\DIVG\vec{W})^+=\max(0,\DIVG\vec{W})$ is the positive part of $\DIVG\vec{W}$, and that $R$ is chosen such that
			\begin{equation}\label{assum:coer.R}
				R\rho_\Gamma > \frac12 \norm[C(\Gamma)^d]{\vec{W}}.
			\end{equation}
		\end{subequations}
		Assume that the solution $\overline{T}$ to \eqref{eq:model} belongs to $C^2(\overline{\Omega})$, and let $T$ be the solution to the scheme \eqref{eq:scheme}. Then,
		\begin{equation}\label{eq:error.estimate}
			\norm[V_h]{T-I_h\overline{T}}\lesssim h \norm[C^2(\overline{\Omega})]{\overline{T}}.
		\end{equation}
	\end{theorem}
	
	\begin{proof} See Appendix \ref{appen:proof}.
	\end{proof}
	
	\begin{remark}[About Assumption \eqref{assum:coer}]\label{rem:assum:coer}
		Assumption \eqref{assum:coer.Omega} imposes a relative smallness only of the \emph{positive} part of $\DIVG\vec{W}$.
		In particular, this assumption holds if $\DIVG\vec{W}\le 0$. Assumption \eqref{assum:coer.R} shows how the user-defined parameter $R$ must be chosen to ensure the stability of the method.
	\end{remark}
	
	\begin{remark}[Regularity assumption on {$\overline{T}$}]
		In most situations, the $C^2$ regularity on $\overline{T}$ can be weakened to an $H^2$ regularity, upon additional technicalities that we do not address here to simplify the exposition. See, e.g., \cite[Section 7.4]{gdm} for lemmas useful for establishing consistency estimates under $H^2$-regularity of the function.
	\end{remark}
	
	\begin{remark}[Assumption \eqref{assum:mesh}]\label{rem:assum:mesh}
		In case Assumption \eqref{assum:mesh} is not satisfied, that is the points $\vec{x}_p$ corresponding to cells that touch $\Gamma$ do not lie on $\Gamma$, the scheme has to be slightly modified the following
		way:
		\begin{itemize}
			\item Additional unknowns on the faces on $\Gamma$ are introduced, so that $V_h$ is changed into
			\begin{align*}
				V_{h}:=\{\varphi=((\varphi_{p})_{p\in \mesh},(\varphi_\sigma)_{\sigma\in\FD\cup\FG},(\varphi_{e})_{e\in \EG})\,:{}&\,\varphi_p\in\R\,,\;\varphi_\sigma\in\R\,,\;\varphi_\sigma=0\mbox{ if $\sigma\in\FD$},\\
				&\varphi_e\in\R\,,\;\varphi_e=0\mbox{ if $e\not\in\EGi$}\}.
			\end{align*}
			A point $\vec{x}_\sigma$ is chosen on each $\sigma\in\FG$ and the interpolant \eqref{eq:interpolant} is extended by setting, for these faces, $\varphi_\sigma=\varphi(\vec{x}_\sigma)$.
			\item The seminorms $\seminorm[V_h,\Omega]{{\cdot}}$ and $\seminorm[V_h,\Gamma]{{\cdot}}$ are modified in the following way: in \eqref{eq:norm.Vh.Omega} the sum is taken over $\sigma\in\F$ with $\varphi_q=\varphi_\sigma$ whenever $\sigma\in\FG$; in \eqref{eq:norm.Vh.Gamma}, $\varphi_p$ is replaced with $\varphi_\sigma$.
			\item Fluxes $\flux{p,\sigma}$ are also considered for $\sigma\in\FG$ and the scheme consists in finding $T\in V_h$ solution to the conservativity equations \eqref{eq:scheme:conserv} and
			\begin{align}
				&\sum_{\sigma\in\F(p)}\flux{p,\sigma}(T)=0,\quad\forall p\in\mesh\,,\\
				&-\flux{p,\sigma}(T)+\sum_{e\in \E(\sigma)}T_{e}\brak{\vec{W}\cdot\vec{n}}{\sigma,e}
				-T_{\sigma}\brak{\DIVG\vec{W}}{\sigma}
				+Rh_\Gamma\sum_{e\in \E(\sigma)}\flux[\Gamma]{\sigma,e}(T) =\iint_{\sigma}g,\quad \forall \sigma\in\FG
				\label{sch.2}
			\end{align}
			where, in \eqref{sch.2}, $p$ is the only cell that has $\sigma$ as face.
			\item The coercivity assumption \eqref{eq:coer.flux.Omega} is changed into
			\[
			\sum_{\sigma\in \F\backslash\FG}\flux{p,\sigma}(\varphi)(\varphi_{p}-\varphi_{q})
			+\sum_{\sigma\in\FG} \flux{p,\sigma}(\varphi)(\varphi_p-\varphi_\sigma)\ge \rho_\Omega \seminorm[V_h,\Omega]{\varphi}^{2}.
			\]
		\end{itemize}
		The analysis performed in Appendix \ref{appen:proof} can then be adapted and leads to the same error
		estimate \eqref{eq:error.estimate}.
	\end{remark}

	\section{Numerical tests}\label{sec:tests}
	
	The numerical tests presented here are obtained using internal fluxes $\flux{p,\sigma}$ corresponding to a corrected version of the ones introduced in \cite{Medla.et.al2018}, and variants. For boundary fluxes $\flux[\Gamma]{\sigma,e}$, used only for stabilisation purposes, we utilise the ones provided by the Hybrid Mimetic Mixed method \cite{DEGH09}.

	\subsection{Description of the scheme}\label{subsec:method}
	
	\subsubsection{Inner fluxes}\label{sec:inner.flux}
	
	We consider a structured, but not necessarily Cartesian, grid of points on $\Omega$. These points are called representative points, as this is where we will look for an approximation of the potential $\overline{T}$. The structured grid assumption means that the representative points can be denoted by $\vec{x}_{i,j,k}$, where $i \in \{0,\ldots, I+1\}$, $j \in \{0,\ldots, J+1\}$, $k \in \{0,\ldots, K+1\}$, and we assume that the extremal points (corresponding to $i=0$, $i=I+1$, $j=0$, $j=J+1$, $k=0$ or $k=K+1$) lie on $\partial\Omega$. We split the set of indices of these extremal points into $\mathcal I_\Gamma=\{(i,j,k)\,:\,k=0\}$ and its complement $\mathcal I_D$, and we assume that $\mathcal I_\Gamma$ corresponds to the points $\vec{x}_{i,j,k}\in\Gamma$, so that $\mathcal I_D$ is the set of indices for the points on the Dirichlet boundary $\partial \Omega\backslash \Gamma$. The points associated with two extremal indices lie on the edges of $\Omega$, whereas those with three extremal indices describe the corners of $\Omega$. Note that $\Omega$ is not necessarily a hexahedron since its ``faces'' may not be planar. We refer to Figs. \ref{fig:cvs} and \ref{fig:nonunif.tess} for illustrations.
	
	For each $(i,j,k)\not\in\mathcal I_D$, a hexahedral finite volume is constructed around $\vec{x}_{i,j,k}$ using the following procedure. Note that points with indices in $\mathcal I_D$ are not associated with control volumes, as they lie on the Dirichlet boundary and they are therefore not associated with unknowns of the scheme.
	
	\begin{itemize}
		\item If $(i,j,k)\in \mathcal  I_{\rm int}:=[2,I-1]\times [2,J-1]\times [2,K-1]$, then setting $A=\{(m,n,o)\in \{-1,0,1\}^3\,:\,|m|+|n|+|o|=3\}$ we define, for $(m,n,o)\in A$, the vertex $\vec{x}_{i,j,k}^{m,n,o}$ as an average of the eight neighbouring points in the grid, one of them being $\vec{x}_{i,j,k}$:
		\begin{equation}
			\vec{x}_{i,j,k}^{m,n,o}=\frac{1}{8}\sum_{(a,b,c)\in B(m,n,o)}\vec{x}_{i+a,j+b,k+c},\label{eq:fv_vertices}
		\end{equation}
		where $B(m,n,o)=\{(m,n,o),$\allowbreak $(m,n,0),$\allowbreak $(m,0,o),$\allowbreak $(m,0,0),$\allowbreak $(0,n,o),$\allowbreak $(0,n,0),$\allowbreak $(0,0,o),$\allowbreak $(0,0,0)\}$.
		The finite volume around $\vec{x}_{i,j,k}$ is then the hexahedron (with possibly non-planar faces) defined by the vertices $\{\vec{x}_{i,j,k}^{m,n,o}\,:\,(m,n,o)\in A\}$. See Fig. \ref{fig:cvs} (left) for an illustration.
		\item If $(i,j,k)\in\mathcal I_\Gamma$, so that $k=0$, and $(i,j)\in [2,I-1]\times [2,J-1]$, we construct four vertices $\vec{x}_{i,j,k}^{m,n,1}$, for $(m,n,1)\in A$, as in \eqref{eq:fv_vertices}. Four more vertices $\vec{x}_{i,j,k}^{m,n,0}$ are constructed by averaging the four neighbouring vertices on $\Gamma$:
		\begin{equation}\label{eq:fv_vertices.face}
			\vec{x}_{i,j,k}^{m,n,0}=\frac{1}{4}\sum_{(a,b)\in C(m,n)}\vec{x}_{i+a,j+b,0},
		\end{equation}
		where $C(m,n)=\{(m,n),(m,0),(0,n),(0,0)\}$. The control volume associated with $\vec{x}_{i,j,0}$ is defined by the eight vertices thus constructed, and we notice that $\vec{x}_{i,j,0}$ lies on one of its faces (the one on $\Gamma$), so that \eqref{assum:mesh} is satisfied.
		\item If $(i,j,k)\not\in (\mathcal I_{\rm int}\cup \mathcal I_{\Gamma})$, $\vec{x}_{i,j,k}$ is associated with a control volume touching the Dirichlet boundary and built from four vertices constructed as in \eqref{eq:fv_vertices} and four other vertices constructed in a similar way as in \eqref{eq:fv_vertices.face}, using representative points on the Dirichlet boundary $\partial\Omega\backslash\Gamma$. See Fig. \ref{fig:cvs} (right).
		\item A similar construction is made for the remaining indices $(i,j,k)$, corresponding to control volumes with an edge along an edge of $\Omega$, or a vertex at one of the corners of $\Omega$; for example, the vertices of the control volumes lying on an edge of $\Omega$ are constructed as the average of two representative points $\vec{x}_{a,b,c}$ with two extremal indices. See Fig. \ref{fig:cvs} (right).
	\end{itemize}
	
	\begin{figure}
		\centering
		\begin{tabular}{ll}
			\includegraphics[width=0.5\textwidth]{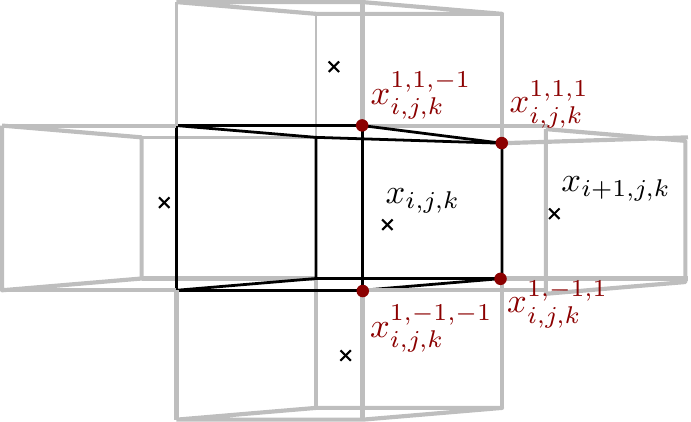} & \includegraphics[width=0.34\textwidth]{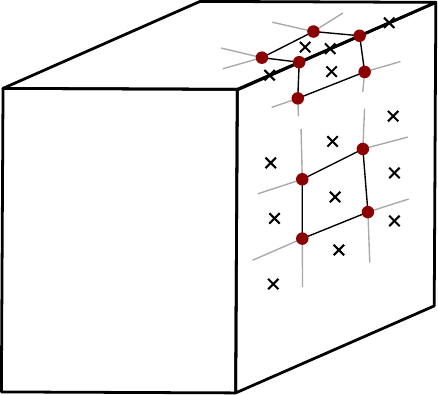}
		\end{tabular}
		\caption{Illustration of the construction of an internal control volume (left), and of the faces and edges of boundary control volumes (right).}
		\label{fig:cvs}
	\end{figure}
	
	A generic finite volume is therefore identified by a triplet $(i,j,k)\not\in\mathcal I_D$. For simplicity and to relate more to the unstructured notations used in Section \ref{sec:scheme}, we denote $(i,j,k)$ by $p$. The point $\vec{x}_{i,j,k}$ associated with $p$ is therefore denoted by $\vec{x}_p$. Any face $\sigma\in\F(p)\backslash \FG$ can be associated with two representative points on each side: $\vec{x}_p$ itself, and $\vec{x}_q$ which might either be associated with a genuine control volume if $\sigma\in \Fi$, or with $\sigma$ itself (as a degenerate cell $q$) if $\sigma\in\FD$. We write $\sigma=\sigma_{pq}$ and notice that $\sigma_{pq}$ may not be planar.
	
	The four vertices $(\vec{x}_{i,j,k}^{m,n,o})_{(m,n,o)}$ of $\sigma_{pq}$ are ordered in a counterclockwise way, respective to the orientation compatible with the outer normal to $p$, and we denote them by $\vec{x}_{pq}^{\oplus}$, $\vec{x}_{pq}^{\boxplus}$, $\vec{x}_{pq}^{\ominus}$, $\vec{x}_{pq}^{\boxminus}$. For $\vec{x}_{pq}^\genexp$ one of these vertices, we let $\mathcal R(\vec{x}_{pq}^\genexp)$ be the set of representative points involved in the construction of $\vec{x}_{pq}^*$; hence, if $\Card{\mathcal R(\vec{x}_{pq}^*)}\in \{8,4,2,1\}$ is the cardinality of $\mathcal R(\vec{x}_{pq}^*)$, we have
	\begin{equation}\label{conv.xpq}
		\vec{x}_{pq}^*=\frac{1}{\Card{\mathcal R(\vec{x}_{pq}^*)}}\sum_{\vec{y}\in \mathcal R(\vec{x}_{pq}^*)}\vec{y}.
	\end{equation}
	
	We define four vectors related to the face $\sigma_{pq}$: the unit vector $\vec{s}_{pq}$ which points from $\vec{x}_p$ to $\vec{x}_q$
	\begin{equation}\label{def:spq}
		\vec{s}_{pq}=\frac{\vec{x}_{q}-\vec{x}_{p}}{|\vec{x}_{q}-\vec{x}_{p}|},
	\end{equation}
	two tangent vectors to the face
	\[
	\vec{t}_{pq}^{\ocircle}=\frac{\vec{x}_{pq}^{\oplus}-\vec{x}_{pq}^{\ominus}}{|\vec{x}_{pq}^{\oplus}-\vec{x}_{pq}^{\ominus}|},\qquad
	\vec{t}_{pq}^{\square}=\frac{\vec{x}_{pq}^{\boxplus}-\vec{x}_{pq}^{\boxminus}}{|\vec{x}_{pq}^{\boxplus}-\vec{x}_{pq}^{\boxminus}|},
	\]
	and
	\[
	\widetilde{\vec{n}}_{pq}=\frac12 (\vec{x}_{pq}^{\oplus}-\vec{x}_{pq}^{\ominus})\times(\vec{x}_{pq}^{\boxplus}-\vec{x}_{pq}^{\boxminus}).
	\]
	Due to the orientation chosen on $\sigma_{pq}$ and the ordering of the vertices of this face, if $\vec{n}_{pq}:\sigma_{pq}\to\R^3$ is the pointwise outer unit normal to $p$ on $\sigma_{pq}$ we  have
	\begin{equation}\label{npq:ave}
		\widetilde{\vec{n}}_{pq}=\iint_{\sigma_{pq}}\vec{n}_{pq}.
	\end{equation}
	
	Since $(\vec{s}_{pq},\vec{t}_{pq}^{\ocircle},\vec{t}_{pq}^{\square})$ form a basis of $\R^3$, there are $\beta_{pq}>0$ and $(\alpha_{pq}^\ocircle,\alpha_{pq}^\square)\in\R^2$ such that
	\begin{equation}\label{npq:rep}
		\widetilde{\vec{n}}_{pq}=|\sigma_{pq}|\left(\frac{1}{\beta_{pq}}\vec{s}_{pq}-\frac{\alpha_{pq}^{\ocircle}}{\beta_{pq}}\vec{t}_{pq}^{\ocircle}-\frac{\alpha_{pq}^{\square}}{\beta_{pq}}\vec{t}_{pq}^{\square}\right).
	\end{equation}
	The numerical fluxes are then given by: for $p\in\mesh$, $\sigma\in\F(p)\backslash\FG$, and $\varphi\in V_h$,
	\begin{equation}
		\label{eq:flux_approx}
		\flux{p,\sigma_{pq}}(\varphi)  =|\sigma_{pq}|\Bigg(\frac{1}{\beta_{pq}}\frac{\varphi_{p}-\varphi_{q}}{d_{pq}}
		-\frac{\alpha_{pq}^{\ocircle}}{\beta_{pq}}\frac{\varphi_{pq}^\oplus-\varphi_{pq}^{\ominus}}{d_{pq}^{\ocircle}} 
		-\frac{\alpha_{pq}^{\square}}{\beta_{pq}}\frac{\varphi_{pq}^{\boxplus}-\varphi_{pq}^{\boxminus}}{d_{pq}^{\boxempty}}\Bigg),
	\end{equation}
	where
	\begin{itemize}
		\item $d_{pq}=|\vec{x}_p-\vec{x}_q|$ (as in Section \ref{sec:scheme}), $d_{pq}^{\ocircle}=|\vec{x}_{pq}^{\oplus}-\vec{x}_{pq}^{\ominus}|$ and $d_{pq}^{\boxempty}=|\vec{x}_{pq}^{\boxplus}-\vec{x}_{pq}^{\boxminus}|$, and
		\item for $*\in\{\oplus,\ominus,\boxplus,\boxminus\}$, each point $\vec{y}\in\mathcal R(\vec{x}_{pq}^\genexp)$ is associated with a genuine or degenerate cell $r$ (possibly an edge or corner of $\Omega$); we then let $\varphi_{\vec{y}}=\varphi_r$ (with $\varphi_r=0$ if $r$ is an edge or corner on $\partial\Omega$) and, following \eqref{conv.xpq}, the secondary unknown $\varphi_{pq}^\genexp$ located at the vertex $\vec{x}_{pq}^\genexp$ is defined by
		\begin{equation}\label{def:varphipq}
			\varphi_{pq}^\genexp=\frac{1}{\Card{\mathcal R(\vec{x}_{pq}^\genexp)}}\sum_{\vec{y}\in \mathcal R(\vec{x}_{pq}^\genexp)}\varphi_{\vec{y}}.
		\end{equation}
	\end{itemize}

\begin{remark}[Correction of the flux in \cite{Medla.et.al2018}]

In \cite{Medla.et.al2018}, a similar flux is defined with right-hand side multiplied by $\frac{|\widetilde{\vec{n}}_{pq}|}{|\sigma_{pq}|}$ in \eqref{eq:flux_approx} -- that is, $\beta_{pq}$ is multiplied by $\frac{|\sigma_{pq}|}{|\widetilde{\vec{n}}_{pq}|}$. The consistency analysis in Section \ref{sec:cons.in.flux} shows that this choice of flux is only consistent if the faces are asymptotically flat (that is, $\frac{|\widetilde{\vec{n}}_{pq}|}{|\sigma_{pq}|}\to 1$ as the mesh size tends to zero). The flux we define above is consistent even if some faces remain non-flat as the mesh size tends to zero.
\end{remark}
	
	\subsubsection{Surface fluxes on $\Gamma$}\label{sec:surface.flux}
	
	We use the fluxes of the Mixed Finite Volumes \cite{DE06}, which is the finite volume presentation of the Hybrid Mimetic Mixed method (see \cite{DEGH09} and \cite[Section 13.2.2]{gdm}). Let $\varphi\in V_h$ and define, for $\sigma\in\FG$ and denoting as usual by $p\in\mesh$ the unique control volume that contains $\sigma$ in its boundary,
	\begin{align}
		\nabla_{\Gamma,\sigma}\varphi={}&\frac{1}{|\sigma|}\sum_{e\in\E(\sigma)}|e|\varphi_e\vec{n}_{\sigma,e},\label{hmm:def.grad}\\
		S_{p,e}(\varphi)={}&\varphi_e-\varphi_p-\nabla_{\Gamma,\sigma}\varphi\cdot (\overline{\vec{x}}_e-\vec{x}_p),\quad\forall e\in\E(\sigma),
		\label{hmm:def.T}
	\end{align}
	where $\overline{\vec{x}}_e=\frac{1}{|e|}\int_e\vec{x}$ is the centre of mass of $e$. Assuming that $\vec{n}_{\sigma,e}$ is constant along $e$ (but see Remark \ref{rem:curved.e} below), the Stokes formula and the definition \eqref{eq:interpolant} of $I_h$ easily show that $\nabla_{\Gamma,\sigma}$ is a consistent approximation of the tangential gradient on $\Gamma$ in the sense that, if $\varphi\in C^2(\Gamma)$,
	\begin{equation}\label{hmm:cons.grad}
		\nabla_{\Gamma,\sigma}I_h\varphi=\frac{1}{|\sigma|}\int_\sigma \nabla_{\Gamma}\varphi.
	\end{equation}
	As a consequence, $S_{p,e}$ can be seen as the remainder of a discrete first order Taylor expansion.
	The HMM fluxes are then defined by: for all $\varphi\in V_h$ and all $\sigma\in\FG$, the family $(\flux[\Gamma]{\sigma,e}(\varphi))_{e\in\E(\sigma)}$ is the unique solution to
	\begin{equation}\label{hmm:def.flux}
		\sum_{e\in\E(\sigma)}\flux[\Gamma]{\sigma,e}(\varphi)(\psi_p-\psi_e)=|\sigma|\nabla_{\Gamma,\sigma}\varphi\cdot
		\nabla_{\Gamma,\sigma}\psi+\sum_{e\in\E(\sigma)}\frac{|e|}{d_{pe}^\bot}S_{p,e}(\varphi)S_{p,e}(\psi)\,,\quad
		\forall \psi\in V_h,
	\end{equation}
	where we recall that $d_{pe}^\bot$ is the orthogonal distance between $\vec{x}_p$ and $e$.
	
	\begin{remark}[Curved edges]\label{rem:curved.e}
		The definition \eqref{hmm:def.grad} is consistent if the only curvature of the edges is due to the curvature of $\Gamma$, that is, the unit normal vector $\vec{n}_{\sigma,e}$ to $\sigma$ on $e$ along $\Gamma$ is constant. However, the HMM remains asymptotically consistent on faces with slightly curved edges \cite{bre-08-06}, in the sense that
		\[
		\left|\vec{n}_{\sigma,e}-\frac{1}{|e|}\int_e \vec{n}_{\sigma,e}\right|=\mathcal O(h_\Gamma).
		\]
	\end{remark}
	
	\subsubsection{Properties of the fluxes}\label{sec:prop.fluxes}
	
	In this section, we show that, upon some mesh regularity assumption (that can be checked in practice during implementation), the inner and surface fluxes described in Sections \ref{sec:inner.flux} and \ref{sec:surface.flux} are coercive and consistent. As a consequence, Theorem \ref{th:error.estimate} applies to the numerical scheme \eqref{eq:scheme} based on these fluxes, and the error estimate \eqref{eq:error.estimate} holds for this scheme.
	
	\medskip
	
	We first define three mesh regularity factors. The first two are required to establish the properties of the inner fluxes (see Appendix), whereas the third one is linked to the properties of the HMM fluxes
	
	\begin{enumerate}
		\item The first regularity factor is related to the faces not lying on $\Gamma$:
		\begin{equation}
			\begin{aligned}
				{\rm reg}_{\mesh,\Omega}:={}&\max\left\{\frac{|\vec{d}_{pq}^\genexp|}{d_{pq}^\genexp}\,:\,
				\sigma_{pq}\in\F\backslash\FG\,,\,
				\genexp\in\{\ominus,\oplus,\boxminus,\boxplus\}\right\}\\
				&+\max\left\{\frac{1}{\left|\det(\vec{s}_{pq},\vec{t}_{pq}^{\ocircle},\vec{t}_{pq}^{\square})\right|}\,:\,
				\sigma_{pq}\in\F\backslash\FG \right\},
			\end{aligned}
			\label{regfac:cons}
		\end{equation}
		where $\vec{d}_{pq}^\genexp=(|\vec{x}_{r}-\vec{x}_{pq}^\genexp|)_{r\in \mathcal R(\vec{x}_{pq}^\genexp)}$ and $d_{pq}^\genexp=d_{pq}^\ocircle$ if $\genexp\in\{\ominus,\oplus\}$, $d_{pq}^\genexp=d_{pq}^\boxempty$ if $\genexp\in\{\boxminus,\boxplus\}$. 
		
		\item The definition of the second regularity factor requires the introduction of a few notations associated to a pair $(\vec{x}_{pq}^\star,r)$, where $\vec{x}_{pq}^\genexp$ is a vertex of a face $\sigma_{pq}\in\F\backslash\FG$ and $r=p$ or $q$. We refer to Figure \ref{fig:neigh} for an illustration of these notations.
		\begin{itemize}
			\item If $\vec{x}_{pq}^\genexp\in\Omega$ is an internal vertex, we let $F_{r,pq}^\genexp$ be the set of the two control volumes that are neighbours of $r$, have $\vec{x}_{pq}^\genexp$ as vertex, but are neither $p$ or $q$. The two control volumes in $F_{r,pq}^\genexp$ have two neighbours in common: $r$ itself, and another control volume that we denote by $e_{r,pq}^\genexp$.
			\item If $\vec{x}_{pq}^\genexp\in \Gamma$, we let $F_{r,pq}^\genexp$ be the set made of the only control volume neighbour of $r$, that has $\vec{x}_{pq}^\genexp$ as vertex, but that is not $p$ or $q$.
			\item If $\vec{x}_{pq}^\genexp$ lies on the Dirichlet boundary $\partial\Omega\backslash \Gamma$, it is the vertex of a face $\sigma\in\F(r)\cap\FD$. We let $F_{r,pq}^\genexp$ be the set made of this face, which is identified to the degenerate control volume $q$.
		\end{itemize}
		\begin{figure}[h!]
			\centering
			\input{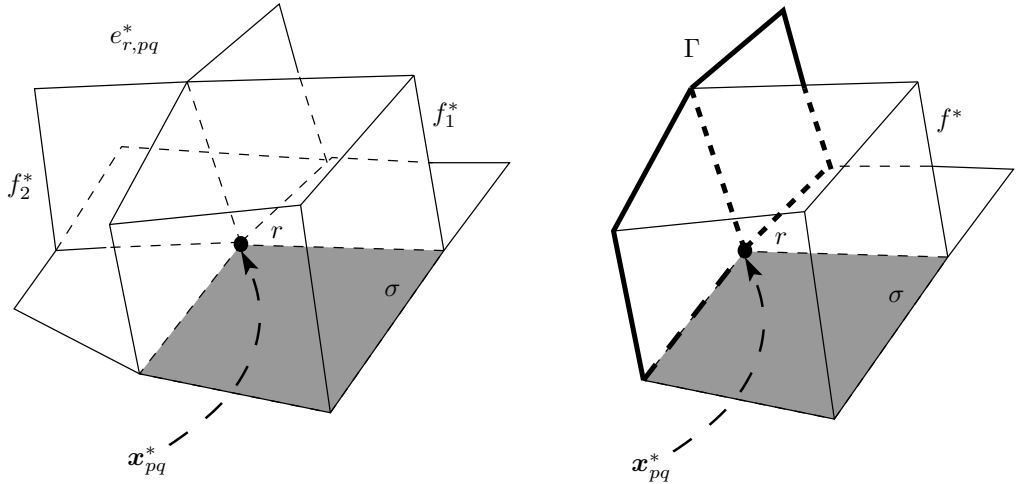}
			\caption{Illustration of the local labels around a face $\sigma=\sigma_{pq}$ and one of its vertices $\vec{x}_{pq}^\genexp$. Left: $\vec{x}_{pq}^\genexp$ is an internal vertex (then $F_{r,pq}^\genexp=\{f_1^\genexp,f_2^\genexp\}$);
				right: $\vec{x}_{pq}^\genexp$ lies on $\Gamma$ (then $F_{r,pq}^\genexp=\{f^\genexp\}$).}\label{fig:neigh}
		\end{figure}
		We then need to know, for a given face $\sigma_{ab}\in\F\backslash\FG$, for which triplet $(p,q,\genexp)$ we have, for some $r=p$ or $q$ and $f\in F^\genexp_{r,pq}$, $\{e^\genexp_{r,pq},f\}=\{a,b\}$ or $\{r,f\}=\{a,b\}$. These triplets are described by the following two sets
		\begin{equation}\label{def:XYab}
			\begin{aligned}
				X_{ab}:=\{(p,q,\genexp)\,:{}&\,\sigma_{pq}\in \F\backslash(\FG\cup\FD)\\
				&\mbox{ and, for some $r\in \{p,q\}$ and $f\in F^\genexp_{r,pq}$, $\{a,b\}=\{e^\genexp_{r,pq},f\}$}\},\\
				Y_{ab} := \{(p,q,\genexp)\,:{}&\,\sigma_{pq}\in \F\backslash(\FG\cup\FD)\\
				&\mbox{ and, for some $r\in \{p,q\}$ and $f\in F^\genexp_{r,pq}$, $\{a,b\}=\{r,f\}$}\}.
			\end{aligned}
		\end{equation}
		The second regularity factor is then $\varrho_{\mesh,\Omega}$, assumed to be $>0$, such that 
		\begin{equation}\label{def:varrho.h}
			\begin{aligned}
				\varrho_{\mesh,\Omega}:=\min\Bigg\{\Bigg[\frac{1}{\beta_{ab}}-\epsilon_{ab}\frac{\alpha_{ab}^{\ocircle}d_{ab}}{2\beta_{ab}d_{ab}^{\ocircle}}-{}&\epsilon_{ab}\frac{\alpha_{ab}^{\square}d_{ab}}{2\beta_{ab}d_{ab}^{\square}}\Bigg]-\frac{1}{16}\sum_{(p,q,\genexp)\in X_{ab}}\zeta_{X,pq}^\genexp\frac{|\sigma_{pq}|d_{ab}\alpha_{pq}^{\diamondsuit}}{|\sigma_{ab}|d_{pq}^{\diamondsuit}\beta_{pq}}\\
				& \quad-\frac{1}{16}\sum_{(p,q,\genexp)\in Y_{ab}}\zeta_{Y,pq}^\genexp\frac{|\sigma_{pq}|d_{ab}\alpha_{pq}^{\diamondsuit}}{|\sigma_{ab}|d_{pq}^{\diamondsuit}\beta_{pq}}\,:\,\sigma_{ab}\in\F\backslash\FG\Bigg\},
			\end{aligned}
		\end{equation}
		where
		\begin{equation}\label{def:eps.zeta}
			\epsilon_{ab}=\left\{\begin{array}{ll}0&\mbox{ if }\sigma_{ab}\in\FD,\\
				1&\mbox{ otherwise },
			\end{array}\right.\qquad
			(\zeta_{X,pq}^\genexp,\zeta_{Y,pq}^\genexp)=\left\{\begin{array}{ll}(1,3)&\mbox{ if }\vec{x}_{pq}^\genexp\in\Omega,\\
				(0,4)&\mbox{ if }\vec{x}_{pq}^\genexp\in\Gamma,\\
				(0,8)&\mbox{ if }\vec{x}_{pq}^\genexp\in\partial\Omega\backslash\Gamma,
			\end{array}\right.
		\end{equation}
		and 
		\begin{equation}\label{def:diamondsuit}
			\diamondsuit=\left\{\begin{array}{ll}\ocircle&\mbox{ if }\genexp\in\{\oplus,\ominus\},\\
				\square&\mbox{ if }\genexp\in\{\boxplus,\boxminus\}.
			\end{array}\right.
		\end{equation}
		
		\item The third regularity factor is
		\begin{equation}
			\label{def:reg.gamma}
			\begin{aligned}
				{\rm reg}_{\mesh,\Gamma}:={}&\max\left\{\frac{{\rm diam}(\sigma)}{d_{pe}^\bot}\,:\,\sigma\in\FG\,,\;e\in\E(\sigma)\right\}\\
				&+\max\left\{\frac{{\rm diam}(\sigma)}{{\rm diam}(\tau)}\,:\,e\in\EGi\,,\;\mbox{$(\sigma,\tau)$  faces on each side of $e$}\right\}.
			\end{aligned}
		\end{equation}
		
	\end{enumerate}
	
	\begin{proposition}[Properties of the fluxes]\label{prop.fluxes}
		The fluxes defined in Sections \ref{sec:inner.flux} and \ref{sec:surface.flux} satisfy Assumption \ref{assum:fluxes} with $\varrho_\Omega=\varrho_{\mesh,\Omega}$, $\Ccons{\Omega}$ depending only on an upper bound of ${\rm reg}_{\mesh}+{\rm reg}_{\mesh,\Omega}$, and $\varrho_\Gamma>0$ and $\Ccons{\Gamma}$ depending only on an upper bound of ${\rm reg}_{\mesh,\Gamma}$.
	\end{proposition}
	
	\begin{proof} See Appendix \ref{appen:prop.method}.
	\end{proof}
	
	\begin{remark}[About the regularity factors]
		Bounding ${\rm reg}_{\mesh,\Omega}$ above imposes the proximity of a vertex $\vec{x}_{pq}^\genexp$ and the representative points $\mathcal R(\vec{x}_{pq}^\genexp)$ involved in its definition, as well as the non-degeneracy of the faces (whose diagonals $\vec{t}_{pq}^\ocircle$ and $\vec{t}_{pq}^\boxempty$ must have a minimal angle) and the transversality of the vector $\vec{x}_p\vec{x}_q$ and the face $\sigma_{pq}$. All these properties are natural given our construction of the control volumes.
		
		The regularity factor ${\rm reg}_{\mesh,\Gamma}$ plays the same role, for the mesh $\FG$ of $\Gamma$, as the regularity factor ${\rm reg}_{\mesh}$ for the mesh $\mesh$ of $\Omega$. See Remark \ref{rem:reg.mesh} for an interpretation of these terms.
		
		Bounding $\varrho_{\mesh,\Omega}$ below imposes that faces that share a common vertex must have comparable measures and diagonal lengths (the terms $\frac{|\sigma_{pq}|}{|\sigma_{ab}|}$ and $\frac{d_{ab}}{d_{pq}^\diamondsuit}$ remain bounded), and that $\vec{s}_{pq}$ is ``not too far'' from the orthogonal direction to $\sigma_{pq}$ (so that, recalling \eqref{npq:rep}, $\beta_{pq}$ remains close to $1$ while $\alpha_{pq}^\ocircle$ and $\alpha_{pq}^\boxempty$ remain small compared to $1$).
		
		All these regularity factors, as well as ${\rm reg}_\mesh$, are easy to numerically evaluate for a given mesh during the implementation. If, as the mesh is refined, these computed factors remain bounded above (for ${\rm reg}_{\mesh}$, ${\rm reg}_{\mesh,\Omega}$ and ${\rm reg}_{\mesh,\Gamma}$) or below (for $\varrho_{\mesh,\Omega}$), then it ensures the robustness and accuracy of the numerical output since the error estimate \eqref{eq:error.estimate} then holds. Note however that these conditions on the regularity factors are merely sufficient, not necessary; the scheme can still, in some cases, converge even if these factors do not remain properly bounded.
	\end{remark}
	
	\subsection{Alternative schemes}
	
	In the numerical tests, we will also present the results using two alternative schemes to the ones described above. The first alternative scheme is similar to \eqref{eq:exact_scheme} in a way that it approximate the oblique derivative as an advection equation on the boundary. It uses an upwind discretisation, instead of a numerically stabilized centred discretisation, for the convective term on the boundary. The second alternative approach is similar to the approximation of inner fluxes \eqref{eq:flux_approx}. It approximates the fluxes through a boundary face on $\Gamma$ by splitting the normal derivative into an oblique component, in the direction $\vec{V}$, and a tangential component to $\Gamma$. Similar splitting, but just on uniform rectangular, radial or spherical grids, was presented in \cite{Macak2012,Macak2014}, where, however, additional points outside domain for treatment of normal derivative were introduced which is made possible with uniform structured grids.

	\subsubsection*{Upwind scheme}
	
	The boundary advection term $\brak{\overline{T}\vec{W}\cdot\vec{n}}{\sigma,e}$ in \eqref{eq:exact_scheme} is here discretised using an upwind approach (which, contrary to \eqref{eq:scheme}, does not require the introduction of numerical diffusion for stabilisation). The resulting scheme has the form
	\begin{equation}
		\begin{aligned}\label{eq:upwind_scheme}
			\sum_{\sigma\in \F(p)\backslash\FG}\flux{p,\sigma}(T)+\sum_{\sigma\in \F(p)\cap\FG}{}&\sum_{e\in \E(\sigma)}\flux[\Gamma,{\rm\scriptsize adv}]{\sigma,e}(T)-\sum_{\sigma\in \F(p)\cap\FG}T_{p}\brak{\DIVG\vec{W}}{\sigma}\\
			& =\sum_{\sigma\in \F(p)\cap\FG}\iint_{\sigma}g,\qquad \forall p\in\mesh,
		\end{aligned}
	\end{equation}
	where the boundary advective numerical flux $\flux[\Gamma,{\rm\scriptsize adv}]{\sigma,e}(T)$, which approximates $\brak{\,\overline{T}\vec{W}\cdot\vec{n}}{\sigma,e}$, is given by
	\[
	\flux[\Gamma,{\rm\scriptsize adv}]{\sigma,e}(T)= \left\{\begin{array}{ll} T_q\brak{\vec{W}\cdot\vec{n}}{\sigma,e}&\quad\mbox{ if $\brak{\vec{W}\cdot\vec{n}}{\sigma,e}<0$},\\
	T_p\brak{\vec{W}\cdot\vec{n}}{\sigma,e}&\quad\mbox{ if $\brak{\vec{W}\cdot\vec{n}}{\sigma,e}\ge 0$}.\end{array}\right.
	\]
	\begin{remark}[Theoretical analysis]\label{rem:analysis.upwind}
		The theoretical analysis of this scheme can be conducted in a similar way as the scheme in Section \ref{subsec:method}, but leads to an $\mathcal O(\sqrt{h})$ theoretical convergence rate (see in particular Remark \ref{rem:centred}).
	\end{remark}
	
	\subsubsection*{Splitting scheme}
	
	Here, the oblique derivative is not recast as a normal component and a boundary advective component. Instead, it is directly used together with a tangential approximation to reconstruct the normal fluxes.
	The resulting scheme has the form
	\begin{equation}
		\label{eq:split_scheme}
		\sum_{\sigma\in \F(p)\backslash\FG}\flux{p,\sigma}(T) + \sum_{\sigma\in\F(p)\cap\FG}\mathcal{F}^{\Gamma}_{p,\sigma}(T) =\sum_{\sigma\in \F(p)\cap\FG}\iint_{\sigma}g,\quad \forall p\in\mesh
	\end{equation}
	where the numerical normal flux $\mathcal{F}^{\Gamma}_{p,\sigma}(T)$, that approximates $\nabla\overline{T}\cdot\vec{n}$, is given by
	\begin{multline}
		\mathcal{F}^{\Gamma}_{p,\sigma}(T) \\
		=|\sigma|\Bigg(\frac{1}{\beta_\sigma}g 
		-\frac{\alpha_\sigma^{\ocircle}}{\beta_\sigma}\frac{\frac{1}{4}\sum_{r\in \mathcal R(x_{\sigma}^{\oplus})}T_{r}-\frac{1}{4}\sum_{r\in \mathcal R(x_\sigma^{\ominus})}T_{r}}{d_\sigma^{\ocircle}} 
		-\frac{\alpha_\sigma^{\square}}{\beta_\sigma}\frac{\frac{1}{4}\sum_{r\in \mathcal R(x_{\sigma}^{\boxplus})}T_{r}-\frac{1}{4}\sum_{r\in \mathcal R(x_\sigma^{\boxminus})}T_{r}}{d_\sigma^{\boxempty}}\Bigg).\label{eq:split_flux_approx}
	\end{multline}
	Here, the coefficients $\beta_{\sigma}$, $\alpha_{\sigma}^{\square}$ and $\alpha_{\sigma}^{\ocircle}$ are given by \eqref{npq:rep} with $\vec{s}_{pq} = \vec{V}(\vec{x}_p)$. They therefore correspond to the decomposition of $\widetilde{\vec{n}}_{pq}$ on the basis $(\vec{V}(\vec{x}_p),\boldsymbol{t}_{\sigma}^{\square},\boldsymbol{t}_{\sigma}^{\ocircle})$. The equation \eqref{eq:split_flux_approx} approximates $\nabla\overline{T}\cdot\vec{n}$ using the oblique derivative $\nabla\overline{T}\cdot\vec{V}=g$ (see \eqref{eq:obliq_BC}) and a tangential component, reconstructed from the boundary values of $T$ using the same principles as in Section \ref{sec:inner.flux}.
	
	\begin{remark}[Theoretical analysis]\label{rem:analysis.split}
		Given the close proximity of the approximation \eqref{eq:split_flux_approx} with the discretisation \eqref{eq:flux_approx}, the ideas developed in Appendix \ref{appen:prop.method} could be adapted to yield an error estimate for this scheme. However, when establishing the coercivity of the method, additional boundary terms would have to be accounted for in the regularity factor \eqref{def:varrho.h}. The scale of these additional negative terms is proportional to how tangential the oblique vector is, and the method fails to be coercive for problems with an oblique field that is too tangential to the boundary of the domain.
	\end{remark}
	
	\subsection{Results}

	To test the proposed methods we present three sets of numerical experiments. For all experiments, the exact solution is chosen to be $\overline{T}(\boldsymbol{x})=\frac{1}{\boldsymbol{x}-\boldsymbol{x}_{0}}$,
	where $\boldsymbol{x}_{0}=(-0.3,-0.2,-0.1)$. The regularity parameters
	\eqref{def:regh}, \eqref{regfac:cons}, \eqref{def:varrho.h}, and
	\eqref{def:reg.gamma} and the Experimental Order of Convergence
	(EOC) are presented. 
	
	\subsubsection{Cubic domain and non-uniform mesh}\label{sec:cubic.domain}
	
	The computational domain $\Omega$ for the first set of experiments is a cube
	with unit edge length. The boundary $\Gamma$, on which the oblique boundary
	condition is prescribed, corresponds to the bottom face of the cube. The mesh is a non-uniform one obtained 
	constructing first a uniform grid with  distance between representative points equal to $h_{u}$, and then moving each 
	point by a random vector $\boldsymbol{r}$ with components in $(-0.15h_{u},0.15h_{u})$.
	Points on $\partial\Omega$ are only moved in a direction tangential to the boundary. Experiments with different oblique vector fields are presented, and the regularity parameters are presented in
	Table \ref{tab:cube_ref_param}. We notice that all parameters remain in a range that makes Theorem \ref{th:error.estimate} and Proposition \ref{prop.fluxes} applicable.
	
	\begin{table}
		\begin{centering}
			\begin{tabular}{|c|c|c|c|c|}
				\hline 
				$h$ & ${\rm reg}_{\mesh}$ & ${\rm reg}_{\mesh,\Gamma}$ & ${\rm reg}_{\mesh,\Omega}$ & $\varrho_{\mesh,\Omega}$\\
				\hline 
				\hline 
				8.511e-01 & 7.311           & 5.536                  & 3.322                  & 6.768e-01                \\ \hline
				3.660e-01 & 7.427           & 6.394                  & 3.279                  & 4.295e-01                \\ \hline
				1.685e-01 & 7.920           & 5.949                  & 3.402                  & 3.508e-01                \\ \hline
				8.309e-02 & 7.879           & 5.967                  & 3.383                  & 2.798e-01                \\ \hline
				4.084e-02 & 8.267           & 6.870                  & 3.510                  & 2.089e-01                \\ \hline
				2.041e-02 & 8.655           & 6.584                  & 3.585                  & 1.669e-01                \\ \hline
				1.014e-02 & 8.356           & 6.854                  & 3.589                  & 1.426e-01                \\ \hline
			\end{tabular}
			\par\end{centering}
		\caption{The regularity parameters \eqref{def:regh}, \eqref{regfac:cons},
			\eqref{def:varrho.h}and \eqref{def:reg.gamma} for a non-uniform mesh of the cube. \label{tab:cube_ref_param}}
	\end{table}
	
	The first experiment, whose results are presented in Table \ref{tab:Cube-constV}, shows the convergence of the method for a constant
	vector field $\vec{V}(\vec{x})=(-1,-1,-1)$. The method \eqref{eq:scheme} displays a first order convergence in $L^2$ and energy norms, which confirms the theoretical prediction of Theorem  \ref{th:error.estimate} and Proposition \ref{prop.fluxes}. The absence of super-convergence in $L^2$ norm is not surprising, as specific Finite Volume methods are only known to super-converge under certain geometric conditions, and to fail to super-converge in some cases \cite{DN17}.
	The rates of convergence for the upwind method \eqref{eq:upwind_scheme} are around $1$ in $L^2$ norm but tend to $1/2$ in $V_{h,\Omega}$ norm, which is expected (see Remark \ref{rem:analysis.upwind}). The splitting method \eqref{eq:split_scheme} shows the best convergence rates:  above second order in $L^2$ norm, and above first order in $V_{h,\Omega}$. 
	
	The second experiment considers the non-constant vector field $\vec{V}(x,y,z)=(x,y,-1)$.
	In this case the surface divergence of $\vec{W}(x,y,z)=(x,y,0)$ (see \eqref{eq:defV}) is $\DIVG\vec{W}(\vec{x})=2$.
	The tests show similar orders of convergence, albeit slightly reduced, as in the experiment
	with a constant vector field; see Table \ref{tab:Cube-divV}. The slight degradation could stem from the fact that the assumption \eqref{assum:coer.Omega} is not fully satisfied on these meshes and with this vector field, or that the asymptotic rate has not been achieved at these mesh sizes.
	
	In the third experiment on the cube we consider a divergence free rotational vector field $\vec{V}(x,y,z)=(-x,z,-1)$. The results in Table \ref{tab:Cube-rotV} show that, here again, the schemes behave in a similar way as with the other two vector fields.
	
	\begin{figure}
		\begin{centering}
			\includegraphics[width=0.6\textwidth,clip]{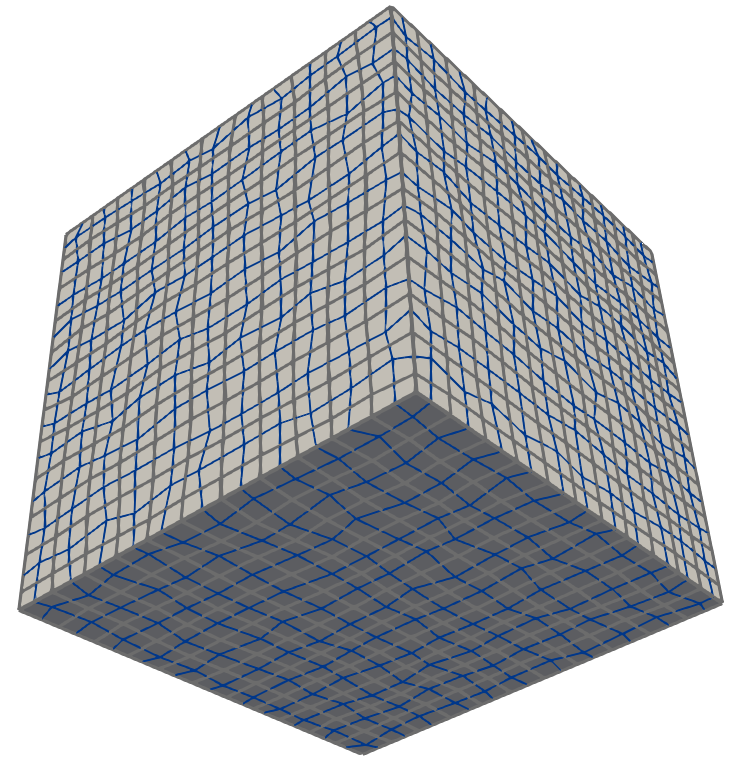}
			\par\end{centering}
		\caption{Non-uniform mesh of the cube}
	\end{figure}

	\begin{table}
		\begin{centering}
			\begin{tabular}{|c||c|c||c|c||c|c||c|c|}
				\hline
				\multicolumn{9}{|c|}{Scheme \eqref{eq:scheme}}\\
				\hline 
				$h$ & $L^2_{\Omega}$ error & EOC & $L^2_{\Gamma}$ error  & EOC & $V_h$ error & EOC & $V_{h,\Gamma}$ error & EOC\\
				\hline 
				\hline
				8.511e-01 & 2.092e-02            &   & 4.600e-02            &   & 1.886e-01   &   & 1.844e-01            &   \\ \hline
				3.660e-01 & 9.412e-03            & 0.946 & 2.634e-02            & 0.660 & 1.022e-01   & 0.726 & 1.462e-01            & 0.275 \\ \hline
				1.685e-01 & 3.922e-03            & 1.128 & 1.270e-02            & 0.940 & 4.767e-02   & 0.982 & 9.298e-02            & 0.583 \\ \hline
				8.309e-02 & 1.958e-03            & 0.982 & 6.756e-03            & 0.893 & 2.475e-02   & 0.927 & 6.358e-02            & 0.537 \\ \hline
				4.084e-02 & 1.003e-03            & 0.942 & 3.570e-03            & 0.898 & 1.294e-02   & 0.913 & 4.248e-02            & 0.567 \\ \hline
				2.041e-02 & 5.155e-04            & 0.959 & 1.867e-03            & 0.934 & 6.661e-03   & 0.957 & 2.678e-02            & 0.665 \\ \hline
				1.014e-02 & 2.560e-04            & 1.000 & 9.360e-04            & 0.986 & 3.287e-03   & 1.009 & 1.587e-02            & 0.747 \\ \hline
			\end{tabular}
\\[.5em]
			\begin{tabular}{|c||c|c||c|c||c|c|}
				\hline
				\multicolumn{7}{|c|}{Scheme \eqref{eq:upwind_scheme}}\\
				\hline 
				$h$ & $L^{2}_{\Omega}$ error & EOC & $L^{2}_{\Gamma}$ error & EOC & $V_{h,\Omega}$ error & EOC\\
				\hline 
				\hline 
				8.491e-01 & 3.224e-02              &   & 7.101e-02              &   & 1.409e-01            &   \\ \hline
				3.683e-01 & 9.117e-03              & 1.512 & 2.530e-02              & 1.236 & 6.281e-02            & 0.967 \\ \hline
				1.688e-01 & 2.972e-03              & 1.436 & 9.999e-03              & 1.190 & 3.070e-02            & 0.917 \\ \hline
				8.303e-02 & 1.237e-03              & 1.236 & 4.555e-03              & 1.108 & 1.898e-02            & 0.677 \\ \hline
				4.192e-02 & 5.513e-04              & 1.182 & 2.147e-03              & 1.101 & 1.260e-02            & 0.599 \\ \hline
				2.045e-02 & 2.579e-04              & 1.058 & 1.036e-03              & 1.014 & 8.562e-03            & 0.538 \\ \hline
				1.018e-02 & 1.233e-04              & 1.058 & 5.075e-04              & 1.023 & 6.030e-03            & 0.502 \\ \hline
			\end{tabular}
\\[.5em]
\begin{tabular}{|c||c|c||c|c||c|c|}
				\hline
				\multicolumn{7}{|c|}{Scheme \eqref{eq:split_scheme}}\\
				\hline 
				$h$ & $L^{2}_{\Omega}$ error & EOC & $L^{2}_{\Gamma}$ error & EOC & $V_{h,\Omega}$ error & EOC\\
				\hline 
				\hline 
				8.659e-01 & 2.154e-02              &   & 3.991e-02              &   & 8.012e-02            &   \\ \hline
				3.692e-01 & 4.645e-03              & 1.800 & 1.186e-02              & 1.424 & 3.579e-02            & 0.945 \\ \hline
				1.691e-01 & 6.920e-04              & 2.437 & 2.228e-03              & 2.140 & 1.187e-02            & 1.413 \\ \hline
				8.246e-02 & 1.606e-04              & 2.035 & 5.954e-04              & 1.838 & 4.569e-03            & 1.330 \\ \hline
				4.066e-02 & 3.457e-05              & 2.172 & 1.353e-04              & 2.096 & 1.836e-03            & 1.290 \\ \hline
				2.046e-02 & 8.271e-06              & 2.083 & 3.127e-05              & 2.133 & 7.615e-04            & 1.282 \\ \hline
				1.029e-02 & 2.035e-06              & 2.039 & 7.885e-06              & 2.004 & 3.502e-04            & 1.130 \\ \hline
			\end{tabular}

	\end{centering}
		\caption{EOC for the non-uniform mesh of the cube with $\boldsymbol{V}(\boldsymbol{x})=(-1,-1,-1)$\label{tab:Cube-constV}}
	\end{table}

	\begin{table}
		\begin{centering}
			\begin{tabular}{|c||c|c||c|c||c|c||c|c|}
				\hline
				\multicolumn{9}{|c|}{Scheme \eqref{eq:scheme}}\\
				\hline 
				$h$ & $L^{2}_{\Omega}$ error & EOC & $L^{2}_{\Gamma}$ error & EOC & $V_{h}$ error & EOC & $V_{h,\Gamma}$ error & EOC\\
				\hline 
				\hline 
				8.478e-01 & 1.508e-02              &   & 2.908e-02              &   & 1.066e-01     &   & 1.021e-01             &          \\ \hline
				3.606e-01 & 9.891e-03              & 0.493 & 2.639e-02              & 0.113 & 8.779e-02     & 0.226 & 1.208e-01             & -0.197422 \\ \hline
				1.694e-01 & 5.583e-03              & 0.757 & 1.671e-02              & 0.605 & 5.182e-02     & 0.697 & 9.645e-02            & 0.298902  \\ \hline
				8.197e-02 & 3.237e-03              & 0.751 & 1.001e-02              & 0.705 & 2.929e-02     & 0.786 & 7.008e-02            & 0.439947  \\ \hline
				4.068e-02 & 1.803e-03              & 0.835 & 5.652e-03              & 0.815 & 1.571e-02     & 0.889 & 4.541e-02            & 0.619154  \\ \hline
				2.032e-02 & 9.617e-04              & 0.905 & 3.036e-03              & 0.895 & 8.098e-03     & 0.954 & 2.699e-02            & 0.749373  \\ \hline
				1.031e-02 & 5.041e-04              & 0.951 & 1.598e-03              & 0.945 & 4.148e-03     & 0.985 & 1.534e-02             & 0.832171  \\ \hline
			\end{tabular}
\\[.5em]
\begin{tabular}{|c||c|c||c|c||c|c|}
				\hline
				\multicolumn{7}{|c|}{Scheme \eqref{eq:upwind_scheme}}\\
				\hline 
				$h$ & $L^{2}_{\Omega}$ error & EOC & $L^{2}_{\Gamma}$ error & EOC & $V_{h,\Omega}$ error & EOC\\
				\hline 
				\hline 
				8.700e-01 & 3.336e-02              &   & 7.199e-02              &   & 1.388e-01            &   \\ \hline
				3.639e-01 & 8.894e-03              & 1.517 & 2.369e-02              & 1.275 & 5.440e-02            & 1.075 \\ \hline
				1.684e-01 & 3.048e-03              & 1.390 & 9.325e-03              & 1.210 & 2.413e-02            & 1.055 \\ \hline
				8.341e-02 & 1.390e-03              & 1.118 & 4.350e-03              & 1.085 & 1.256e-02            & 0.929 \\ \hline
				4.114e-02 & 6.835e-04              & 1.004 & 2.169e-03              & 0.984 & 7.078e-03            & 0.811 \\ \hline
				2.055e-02 & 3.410e-04              & 1.002 & 1.085e-03              & 0.997 & 4.233e-03            & 0.740 \\ \hline
				1.015e-02 & 1.698e-04              & 0.987 & 5.414e-04              & 0.985 & 2.700e-03            & 0.637 \\ \hline
			\end{tabular}
\\[.5em]
\begin{tabular}{|c||c|c||c|c||c|c|}
				\hline
				\multicolumn{7}{|c|}{Scheme \eqref{eq:split_scheme}}\\
				\hline 
				$h$ & $L^{2}_{\Omega}$ error & EOC & $L^{2}_{\Gamma}$ error & EOC & $V_{h,\Omega}$ error & EOC\\
				\hline 
				\hline 
				8.455e-01 & 1.617e-02              &   & 3.542e-02              &   & 7.634e-02            &   \\ \hline
				3.630e-01 & 2.437e-03              & 2.237 & 6.801e-03              & 1.952 & 1.956e-02            & 1.610 \\ \hline
				1.722e-01 & 7.058e-04              & 1.662 & 1.880e-03              & 1.725 & 8.495e-03            & 1.119 \\ \hline
				8.220e-02 & 1.720e-04              & 1.908 & 4.538e-04              & 1.921 & 3.432e-03            & 1.225 \\ \hline
				4.063e-02 & 3.849e-05              & 2.125 & 1.034e-04              & 2.099 & 1.455e-03            & 1.217 \\ \hline
				2.032e-02 & 9.436e-06              & 2.029 & 2.561e-05              & 2.014 & 6.904e-04            & 1.076 \\ \hline
				1.022e-02 & 2.326e-06              & 2.039 & 6.338e-06              & 2.033 & 3.302e-04            & 1.074 \\ \hline
			\end{tabular}
			\par\end{centering}
		\caption{EOC for the non-uniform mesh of the cube with $\boldsymbol{V}(x,y,z)=(x,y,-1)$\label{tab:Cube-divV}}
	\end{table}

	\begin{table}
		\begin{centering}
			\begin{tabular}{|c||c|c||c|c||c|c||c|c|}
				\hline
				\multicolumn{9}{|c|}{Scheme \eqref{eq:scheme}}\\
				\hline 
				$h$ & $L^{2}_{\Omega}$ error & EOC & $L^{2}_{\Gamma}$ error & EOC & $V_{h}$ error & EOC & $V_{h,\Gamma}$ error & EOC\\
				\hline 
				\hline 
				8.594e-01 & 1.402e-02              &   & 3.084e-02              &   & 1.103e-01     &   & 1.070e-01            &    \\ \hline
				3.651e-01 & 7.731e-03              & 0.695 & 2.187e-02              & 0.401 & 7.795e-02     & 0.405 & 1.084e-01            & -0.015 \\ \hline
				1.703e-01 & 4.769e-03              & 0.633 & 1.455e-02              & 0.534 & 4.839e-02     & 0.625 & 9.136e-02            & 0.224  \\ \hline
				8.271e-02 & 2.620e-03              & 0.828 & 8.360e-03              & 0.767 & 2.636e-02     & 0.840 & 6.458e-02            & 0.480  \\ \hline
				4.070e-02 & 1.389e-03              & 0.894 & 4.520e-03              & 0.867 & 1.347e-02     & 0.947 & 4.045e-02            & 0.659  \\ \hline
				2.037e-02 & 7.382e-04              & 0.913 & 2.428e-03              & 0.897 & 6.985e-03     & 0.948 & 2.454e-02            & 0.721  \\ \hline
				1.018e-02 & 3.768e-04              & 0.969 & 1.248e-03              & 0.960 & 3.490e-03     & 1.000 & 1.378e-02            & 0.832  \\ \hline
				
			\end{tabular}
\\[0.5em]
			\begin{tabular}{|c||c|c||c|c||c|c|}
				\hline
				\multicolumn{7}{|c|}{Scheme \eqref{eq:upwind_scheme}}\\
				\hline 
				$h$ & $L^{2}_{\Omega}$ error & EOC & $L^{2}_{\Gamma}$ error & EOC & $V_{h,\Omega}$ error & EOC\\
				\hline 
				\hline 
				8.524e-01 & 1.366e-02              &   & 1.930e-02              &   & 6.906e-02            &   \\ \hline
				3.670e-01 & 2.301e-03              & 2.114 & 6.247e-03              & 1.339 & 2.381e-02            & 1.264 \\ \hline
				1.679e-01 & 9.552e-04              & 1.124 & 3.374e-03              & 0.787 & 1.730e-02            & 0.408 \\ \hline
				8.357e-02 & 3.550e-04              & 1.419 & 1.483e-03              & 1.178 & 8.874e-03            & 0.957 \\ \hline
				4.078e-02 & 1.550e-04              & 1.155 & 6.917e-04              & 1.063 & 5.654e-03            & 0.628 \\ \hline
				2.026e-02 & 7.428e-05              & 1.051 & 3.290e-04              & 1.062 & 3.727e-03            & 0.595 \\ \hline
				1.035e-02 & 3.648e-05              & 1.059 & 1.585e-04              & 1.088 & 2.510e-03            & 0.588 \\ \hline
			\end{tabular}
\\[0.5em]
\begin{tabular}{|c||c|c||c|c||c|c|}
				\hline
				\multicolumn{7}{|c|}{Scheme \eqref{eq:split_scheme}}\\
				\hline 
				$h$ & $L^{2}_{\Omega}$ error & EOC & $L^{2}_{\Gamma}$ error & EOC & $V_{h,\Omega}$ error & EOC\\
				\hline 
				\hline 
				8.668e-01 & 8.610e-03              &   & 1.747e-02              &   & 4.397e-02            &   \\ \hline
				3.625e-01 & 1.853e-03              & 1.763 & 3.020e-03              & 2.013 & 1.659e-02            & 1.118 \\ \hline
				1.689e-01 & 5.089e-04              & 1.692 & 9.059e-04              & 1.576 & 7.446e-03            & 1.049 \\ \hline
				8.294e-02 & 1.174e-04              & 2.063 & 2.719e-04              & 1.693 & 3.431e-03            & 1.090 \\ \hline
				4.070e-02 & 2.681e-05              & 2.075 & 6.039e-05              & 2.113 & 1.458e-03            & 1.202 \\ \hline
				2.076e-02 & 6.741e-06              & 2.051 & 1.415e-05              & 2.155 & 6.850e-04            & 1.122 \\ \hline
				1.023e-02 & 1.644e-06              & 1.994 & 3.515e-06              & 1.968 & 3.287e-04            & 1.037 \\ \hline
			\end{tabular}
			\par\end{centering}
		\caption{EOC for the non-uniform mesh of the cube with $\boldsymbol{V}(x,y,z)=(-x,z,-1)$\label{tab:Cube-rotV}}
	\end{table}

	\subsubsection{Tesseroid domain with non-planar $\Gamma$}\label{sec:tesseroid}
	
	The next experiments are run on a computational domain with a non-planar
	boundary $\Gamma$. The discrete computational domain then does not exactly match $\Omega$, and vertices of the boundary faces do not have to lie on the boundary $\Gamma$. Moreover, in this construction, the tangent space
	to $\Gamma$ is not well defined everywhere so the co-normal $\vec{n}_{\sigma,e}$
	is not well defined either in the Eq. \eqref{eq:scheme}. In this case we approximate the normal vector
$\vec{n}_{\sigma,e}$ by the normalised version of $\left(\frac{\vec{N}_{p}+\vec{N}_{q}}{2}\right)\times\vec{e}$,
	where the vector $\vec{N}_{p}$ is a normal to the face $\sigma$,
	the vector $\vec{N}_{q}$ is normal to the neighbouring face on the other side of $e$,
	and the vector $\vec{e}$ is a tangent vector to the edge $e$, chosen
	such that $\vec{n}_{\sigma,e}$ is an outward normal to $\sigma$.
	
	The experiments are performed on a non-uniform mesh of the tesseroid 
	\[
	\Omega=\left\{(r\sin(u)\cos(v), r\sin(u)\sin(v), r\cos(u))\,:\,r\in(1,2)\,,\; u\in\left(\frac{3\pi}{8},\frac{5\pi}{8}\right)\,,\;v\in\left(0,\frac{\pi}{4}\right)\right\}.
	\]
	See Fig. \ref{fig:nonunif.tess} for an illustration. The oblique boundary condition is prescribed on the non-planar face corresponding to $r=1$:
	\[
	\Gamma=\left\{(\sin(u)\cos(v),\sin(u)\sin(v),\cos(u))\,:\,u\in\left(\frac{3\pi}{8},\frac{5\pi}{8}\right)\,,\;
	v\in\left(0,\frac{\pi}{4}\right)\right\}.
	\]
	The regularity parameters of the considered meshes are presented in
	Table \ref{tab:sph_ref_param}. As can be seen there, the regularity factor $\varrho_{\mesh,\Omega}$ seems to degenerate as the mesh size is reduced, indicating that the condition that ensure the coercivity of the inner fluxes (see Proposition \ref{prop.fluxes}) might not hold -- which does not necessarily mean that the scheme actually fails to be coercive or to converge, since this is only a \emph{sufficient} condition.
	\begin{table}
		\begin{centering}
			\begin{tabular}{|c|c|c|c|c|}
				\hline 
				$h$ & ${\rm reg}_{\mesh}$ & ${\rm reg}_{\mesh,\Omega}$ & ${\rm reg}_{\mesh,\Gamma}$ & $\varrho_{\mesh,\Omega}$\\
				\hline 
				\hline 
				1.092 & 2.538e+01 & 5.468 & 3.296 & 6.074e-01 \\ \hline
				4.981e-01 & 1.336e+01 & 5.970 & 3.468 & 4.051e-01 \\ \hline
				2.375e-01 & 1.045e+01 & 6.228 & 3.297 & 3.691e-01 \\ \hline
				1.158e-01 & 1.017e+01 & 6.352 & 3.442 & 2.463e-01 \\ \hline
				5.733e-02 & 0.998e+01 & 6.324 & 3.426 & 1.460e-01 \\ \hline
				2.847e-02 & 1.029e+01 & 6.582 & 3.598 & 1.392e-01 \\ \hline
				1.454e-02 & 1.032e+01 & 6.761 & 3.580 & 8.204e-02 \\ \hline
			\end{tabular}
			\par\end{centering}
		\caption{The regularity parameters \eqref{def:regh}, \eqref{regfac:cons},
			\eqref{def:varrho.h} and \eqref{def:reg.gamma} for a non-uniform mesh of a tesseroid  mesh.
			\label{tab:sph_ref_param}}
	\end{table}
	The tests present the convergence of the methods for a non-constant
	vector field $\boldsymbol{V}(\boldsymbol{x})=(0.3,0.2,0.1)-\boldsymbol{x}$.
	The results presented in Table \ref{tab:Sph-rotV-2} are similar to the ones obtained in Section \ref{sec:cubic.domain}, with perhaps slightly better rates of convergence across the board. In any case, the apparent decay of the regularity factor $\varrho_{\mesh,\Omega}$ does not seem to negatively impact the convergence of the schemes.
	
	\begin{figure}
		\begin{centering}
			\includegraphics[width=0.7\textwidth,clip]{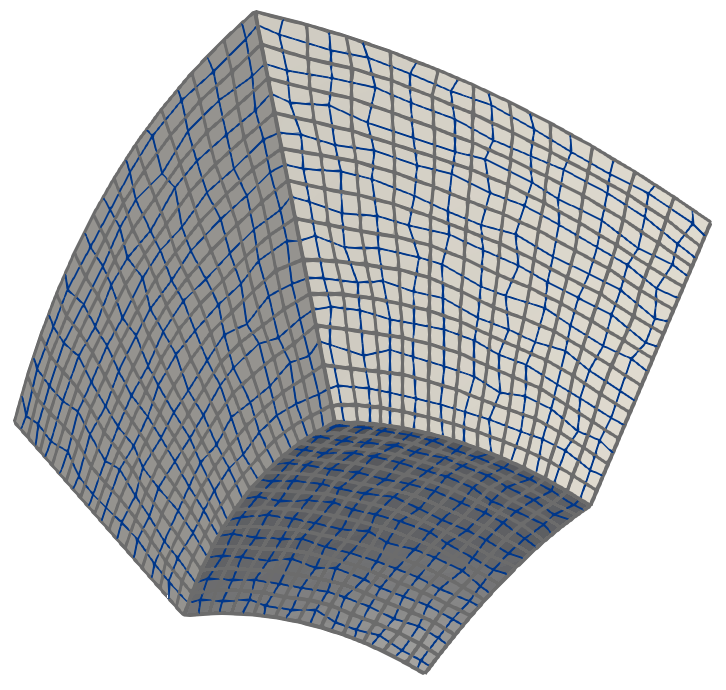}
			\par\end{centering}
		\caption{Non-uniform mesh of a tesseroid}
		\label{fig:nonunif.tess}
	\end{figure}
	
	\begin{table}
		\begin{centering}
			\begin{tabular}{|c||c|c||c|c||c|c||c|c|}
				\hline
				\multicolumn{9}{|c|}{Scheme \eqref{eq:scheme}}\\
				\hline 
				$h$ & $L^{2}_{\Omega}$ error & EOC & $L^{2}_{\Gamma}$ error & EOC & $V_{h}$ error & EOC & $V_{h,\Gamma}$ error & EOC\\
				\hline 
				\hline 
				1.092 & 7.054e-03              &   & 9.644e-03              &   & 3.776e-02     &   & 3.620e-02            &   \\ \hline
				4.981e-01 & 1.307e-03              & 2.148 & 2.189e-03              & 1.890 & 9.516e-03     & 1.756 & 1.205e-02            & 1.402 \\ \hline
				2.375e-01 & 3.656e-04              & 1.720 & 7.681e-04              & 1.414 & 3.249e-03     & 1.451 & 5.267e-03            & 1.117 \\ \hline
				1.158e-01 & 1.294e-04              & 1.446 & 3.335e-04              & 1.161 & 1.322e-03     & 1.252 & 2.698e-03            & 0.931 \\ \hline
				5.733e-02 & 5.493e-05              & 1.219 & 1.578e-04              & 1.065 & 5.788e-04     & 1.175 & 1.400e-03            & 0.932 \\ \hline
				2.847e-02 & 2.518e-05              & 1.114 & 7.580e-05              & 1.047 & 2.643e-04     & 1.119 & 7.200e-04            & 0.950 \\ \hline
				1.454e-02 & 1.248e-05              & 1.044 & 3.843e-05              & 1.011 & 1.285e-04     & 1.073 & 3.805e-04            & 0.949 \\ \hline
			\end{tabular}
\\[.5em]
\begin{tabular}{|c||c|c||c|c||c|c|}
				\hline
				\multicolumn{7}{|c|}{Scheme \eqref{eq:upwind_scheme}}\\
				\hline 
				$h$ & $L^{2}_{\Omega}$ error & EOC & $L^{2}_{\Gamma}$ error & EOC & $V_{h}$ error & EOC\\
				\hline 
				\hline 
				1.075 & 5.920e-03              &   & 5.346e-03              &   & 2.111e-02     &   \\ \hline
				4.783e-01 & 1.184e-03              & 1.989 & 1.474e-03              & 1.592 & 6.483e-03     & 1.458 \\ \hline
				2.326e-01 & 2.935e-04              & 1.934 & 5.223e-04              & 1.439 & 2.464e-03     & 1.342 \\ \hline
				1.156e-01 & 9.156e-05              & 1.665 & 2.141e-04              & 1.275 & 1.055e-03     & 1.213 \\ \hline
				5.766e-02 & 3.433e-05              & 1.411 & 9.374e-05              & 1.188 & 4.987e-04     & 1.077 \\ \hline
				2.826e-02 & 1.458e-05              & 1.201 & 4.358e-05              & 1.074 & 2.661e-04     & 0.880 \\ \hline
				1.435e-02 & 6.738e-06              & 1.139 & 2.099e-05              & 1.078 & 1.569e-04     & 0.779 \\ \hline
			\end{tabular}
\\[.5em]
\begin{tabular}{|c||c|c||c|c||c|c|}
				\hline
				\multicolumn{7}{|c|}{Scheme \eqref{eq:split_scheme}}\\
				\hline 
				$h$ & $L^{2}_{\Omega}$ error & EOC & $L^{2}_{\Gamma}$ error & EOC & $V_{h}$ error & EOC\\
				\hline 
				\hline 
				1.073 & 5.384e-03              &   & 5.121e-03              &   & 1.800e-02     &   \\ \hline
				4.950e-01 & 1.073e-03              & 2.084 & 1.114e-03              & 1.970 & 6.262e-03     & 1.364 \\ \hline
				2.362e-01 & 2.165e-04              & 2.162 & 2.155e-04              & 2.220 & 2.093e-03     & 1.481 \\ \hline
				1.155e-01 & 5.076e-05              & 2.028 & 4.964e-05              & 2.053 & 8.340e-04     & 1.286 \\ \hline
				5.767e-02 & 1.247e-05              & 2.022 & 1.277e-05              & 1.954 & 3.506e-04     & 1.248 \\ \hline
				2.845e-02 & 3.046e-06              & 1.994 & 3.107e-06              & 2.001 & 1.574e-04     & 1.133 \\ \hline
				1.430e-02 & 7.544e-07              & 2.029 & 7.650e-07              & 2.038 & 7.386e-05     & 1.100 \\ \hline
			\end{tabular}
			\par\end{centering}
		\caption{EOC for a non-uniform mesh of a tesseroid with $\boldsymbol{V}(\boldsymbol{x})=(0.3,0.2,0.1)-\boldsymbol{x}$\label{tab:Sph-rotV-2}}
	\end{table}
	
	\subsubsection{Spherical section domain with perturbed bottom $\Gamma$}
	
	This series of experiments is performed on a section of a spherical domain,
	with a perturbed bottom boundary $\Gamma$ (see Fig. \ref{fig:pert.ball}):
	\[
	\begin{aligned}
	\Omega=\Big\{{}&\big[(1+0.04(2 - r)(\sin(10\ u) + \sin(10\ v)))\sin(u)\cos(v),\\
	&\,(1+0.04(2 - r)(\sin(10\ u) + \sin(10\ v)))\sin(u)\sin(v),\\
	&\,(1+0.04(2 - r)(\sin(10\ u) + \sin(10\ v)))\cos(u)\big]\,:\\
	&\qquad\qquad\qquad\qquad\qquad r\in(1,2)\,,\;u\in\left(\frac{3\pi}{8},\frac{5\pi}{8}\right)\,,\;v\in\left(0,\frac{\pi}{4}\right)\Big\},\\
	\Gamma=\Big\{{}&\big[(1+0.04(\sin(10\ u) + \sin(10\ v)))\sin(u)\cos(v),\\
	&\,(1+0.04(\sin(10\ u) + \sin(10\ v)))\sin(u)\sin(v),\\
	&\,(1+0.04(\sin(10\ u) + \sin(10\ v)))\cos(u)\big]\,:\\
	&\qquad\qquad\qquad\qquad\qquad\qquad u\in\left(\frac{3\pi}{8},\frac{5\pi}{8}\right)\,,\;v\in\left(0,\frac{\pi}{4}\right)\Big\}.
	\end{aligned}
	\]
	The regularity parameters for the considered meshes are presented in
	Table \ref{tab:EarthLike_ref_param}. The coercivity constant $\varrho_{\mesh,\Omega}$ is worse as in the tesseroid case, as it becomes negative. However, once again, since a lower bound on this constant is only a sufficient condition for the theoretical analysis, this does not mean that the schemes fail to converge, as the numerical results will show. We take the non-constant vector
	field $\boldsymbol{V}(\boldsymbol{x})=(0.3,0.2,0.1)-\boldsymbol{x}$.
	Table \ref{tab:rotV-1} shows that all three schemes behave in a similar way as in the previous tests of Sections \ref{sec:cubic.domain} and \ref{sec:tesseroid}. This indicates that our coercivity analysis (based on $\varrho_{\mesh,\Omega}$) is actually a bit too conservative regarding the robustness range of the discretisations.
	
	\begin{table}
		\begin{centering}
			\begin{tabular}{|c|c|c|c|c|}
				\hline 
				$h$ & ${\rm reg}_{\mesh}$ & ${\rm reg}_{\mesh,\Omega}$ & ${\rm reg}_{\mesh,\Gamma}$ & $\varrho_{\mesh,\Omega}$\\
				\hline 
				\hline 
				1.089 & 3.228e+01           & 5.017                  & 3.571                & 3.068e-01                  \\ \hline
				4.928e-01 & 1.896e+01           & 6.270                  & 3.551                & -6.303e-01                 \\ \hline
				2.299e-01 & 1.207e+01           & 6.259                  & 3.777                & -9.516e-01                 \\ \hline
				1.159e-01 & 1.120e+01           & 6.819                  & 3.838                & -1.138                 \\ \hline
				5.813e-02 & 1.083e+01           & 7.246                  & 3.904                & -1.524                 \\ \hline
				2.867e-02 & 1.114e+01           & 6.890                  & 4.253                & -1.403                 \\ \hline
				1.462e-02 & 1.118e+01           & 6.914                  & 4.232                & -1.642                 \\ \hline
				
			\end{tabular}
			\par\end{centering}
		\caption{The regularity parameters \eqref{def:regh}, \eqref{regfac:cons},
			\eqref{def:varrho.h} and \eqref{def:reg.gamma} for a non-uniform mesh of a section of a perturbed ball. \label{tab:EarthLike_ref_param}}
	\end{table}
	
	\begin{figure}
		\begin{centering}
			\includegraphics[width=0.7\textwidth,clip]{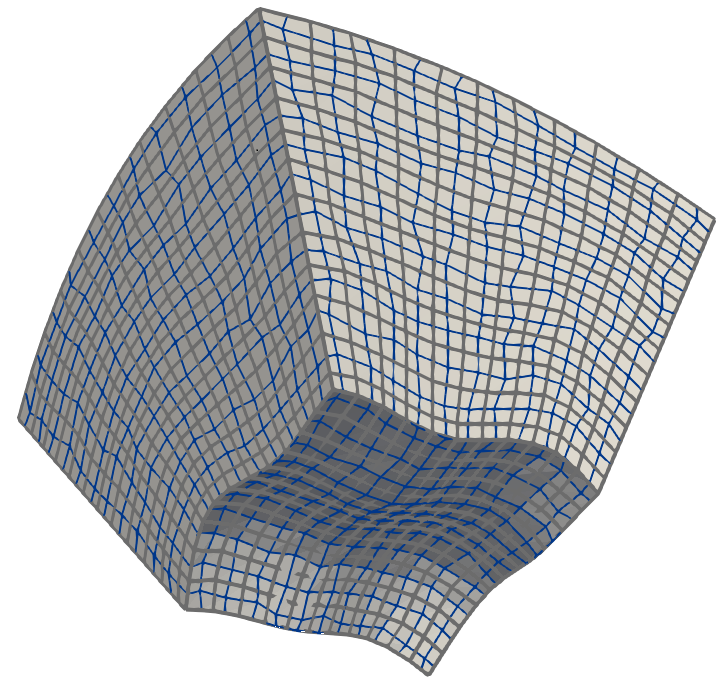}
			\par\end{centering}
		\caption{Section of a perturbed ball}
		\label{fig:pert.ball}
	\end{figure}
	\begin{table}
		\begin{centering}
			\begin{tabular}{|c||c|c||c|c||c|c||c|c|}
				\hline
				\multicolumn{9}{|c|}{Scheme \eqref{eq:scheme}}\\
				\hline 
				$h$ & $L^{2}_{\Omega}$ error & EOC & $L^{2}_{\Gamma}$ error & EOC & $V_{h}$ error & EOC & $V_{h,\Gamma}$ error & EOC\\
				\hline 
				\hline 
				1.089 & 7.577e-03              &   & 1.212e-02              &   & 5.851e-02     &   & 6.169e-02            &    \\ \hline
				4.928e-01 & 2.595e-03              & 1.351 & 6.735e-03              & 0.740 & 4.232e-02     & 0.408 & 7.120e-02            & -0.180 \\ \hline
				2.299e-01 & 1.519e-03              & 0.7021 & 5.123e-03              & 0.358 & 2.766e-02     & 0.557 & 6.351e-02            & 0.150  \\ \hline
				1.159e-01 & 8.678e-04              & 0.817 & 3.163e-03              & 0.704 & 1.465e-02     & 0.928 & 4.295e-02            & 0.571  \\ \hline
				5.813e-02 & 4.763e-04              & 0.868 & 1.776e-03              & 0.835 & 7.250e-03     & 1.019 & 2.550e-02            & 0.754  \\ \hline
				2.867e-02 & 2.548e-04              & 0.885 & 9.569e-04              & 0.875 & 3.594e-03     & 0.993 & 1.433e-02            & 0.815  \\ \hline
				1.462e-02 & 1.335e-04              & 0.959 & 5.028e-04              & 0.955 & 1.793e-03     & 1.032 & 7.777e-03            & 0.907  \\ \hline
				
			\end{tabular}
\\[0.5em]
\begin{tabular}{|c||c|c||c|c||c|c|}
				\hline
				\multicolumn{7}{|c|}{Scheme \eqref{eq:upwind_scheme}}\\
				\hline 
				$h$ & $L^{2}_{\Omega}$ error & EOC & $L^{2}_{\Gamma}$ error & EOC & $V_{h}$ error & EOC\\
				\hline 
				\hline 
				1.120 & 6.715e-03              &   & 6.435e-03              &   & 2.482e-02     &   \\ \hline
				4.881e-01 & 1.643e-03              & 1.695 & 3.335e-03              & 0.791 & 1.269e-02     & 0.807 \\ \hline
				2.329e-01 & 7.427e-04              & 1.073 & 2.294e-03              & 0.505 & 8.068e-03     & 0.612 \\ \hline
				1.142e-01 & 3.565e-04              & 1.030 & 1.256e-03              & 0.845 & 4.444e-03     & 0.836 \\ \hline
				5.749e-02 & 1.717e-04              & 1.065 & 6.353e-04              & 0.993 & 2.336e-03     & 0.936 \\ \hline
				2.890e-02 & 8.426e-05              & 1.035 & 3.175e-04              & 1.008 & 1.245e-03     & 0.915 \\ \hline
				1.437e-02 & 4.168e-05              & 1.008 & 1.582e-04              & 0.997 & 6.884e-04     & 0.848 \\ \hline
			\end{tabular}
\\[0.5em]
\begin{tabular}{|c||c|c||c|c||c|c|}
				\hline
				\multicolumn{7}{|c|}{Scheme \eqref{eq:split_scheme}}\\
				\hline 
				$h$ & $L^{2}_{\Omega}$ error & EOC & $L^{2}_{\Gamma}$ error & EOC & $V_{h}$ error & EOC\\
				\hline 
				\hline 
				1.086 & 6.387e-03              &   & 5.968e-03              &   & 2.235e-02     &   \\ \hline
				4.943e-01 & 1.229e-03              & 2.093 & 1.479e-03              & 1.772 & 7.613e-03     & 1.368 \\ \hline
				2.332e-01 & 2.756e-04              & 1.991 & 3.978e-04              & 1.749 & 2.536e-03     & 1.464 \\ \hline
				1.165e-01 & 6.656e-05              & 2.047 & 1.067e-04              & 1.895 & 9.496e-04     & 1.415 \\ \hline
				5.864e-02 & 1.597e-05              & 2.079 & 2.622e-05              & 2.045 & 3.870e-04     & 1.307 \\ \hline
				2.925e-02 & 3.894e-06              & 2.028 & 6.502e-06              & 2.005 & 1.739e-04     & 1.150 \\ \hline
				1.441e-02 & 9.641e-07              & 1.972 & 1.616e-06              & 1.967 & 8.209e-05     & 1.060 \\ \hline
			\end{tabular}
			\par\end{centering}
		\caption{EOC for $\Omega$ the section of a perturbed ball. $\boldsymbol{V}(\boldsymbol{x})=(0.3,0.2,0.1)-\boldsymbol{x}$\label{tab:rotV-1}}
	\end{table}
	
	\subsubsection{Cubic domain, almost tangential vector field $\vec{V}$}
	
	In this final series of numerical experiments with an analytical solution, we show the advantage of the proposed scheme \eqref{eq:scheme} and of the upwind scheme \eqref{eq:upwind_scheme} over the splitting scheme \eqref{eq:split_scheme}. The computational domain $\Omega$ is the unit cube, with $\Gamma$ being its bottom. We take the vector field $\vec{V} = (11.4301,0, -1)$, which corresponds to the outer normal on $\Gamma$ rotated by $85°$; this vector field is therefore almost tangential to the boundary. The tests are run on uniform meshes.
	
	Table \ref{tab:tangV-advDif} presents the EOC for the proposed scheme and upwind scheme.   The rates are sometimes degraded compared to the previous tests, but there is a clear convergence.
	
	For the splitting scheme \eqref{eq:split_scheme}, the fact that $\vec{V}$ is almost tangential leads to very large
	values of $\frac{\alpha_\sigma^\genexp}{\beta_\sigma}$ in \eqref{eq:split_flux_approx}. As a consequence, the negative coefficients in the coercivity factor are too large to be controlled by the positive coefficients; the scheme really becomes non-coercive and unstable, and the BiCGStab algorithm used to solve the system fails. This breakdown of a numerical method is probably the worst situation that one wants to avoid in practice, which indicates that in severely oblique situations the proposed new methods \eqref{eq:scheme} and \eqref{eq:upwind_scheme} should be preferred, despite yielding sometimes reduced rates of convergence.
	
	\begin{table}
		\begin{centering}
			\begin{tabular}{|c||c|c||c|c||c|c||c|c|}
				\hline
				\multicolumn{9}{|c|}{Scheme \eqref{eq:scheme}}\\
				\hline 
				$h$ & $L^{2}_{\Omega}$ error & EOC & $L^{2}_{\Gamma}$ error & EOC & $V_{h}$ error & EOC & $V_{h,\Gamma}$ error & EOC\\
				\hline 
				\hline 
				8.605e-01 & 2.247e-02              &   & 4.945e-02              &   & 1.686e-01     &   & 1.610e-01            &    \\ \hline
				3.606e-01 & 1.537e-02              & 0.436 & 4.232e-02              & 0.179 & 1.437e-01     & 0.184 & 2.014e-01            & -0.257 \\ \hline
				1.692e-01 & 1.088e-02              & 0.455 & 3.301e-02              & 0.328 & 1.060e-01     & 0.401 & 2.012e-01            & 0.001 \\ \hline
				8.347e-02 & 6.582e-03              & 0.712 & 2.109e-02              & 0.634 & 6.669e-02     & 0.655 & 1.647e-01            & 0.283  \\ \hline
				4.054e-02 & 3.756e-03              & 0.776 & 1.255e-02              & 0.718 & 4.010e-02     & 0.704 & 1.248e-01            & 0.384  \\ \hline
				2.043e-02 & 2.046e-03              & 0.886 & 7.078e-03              & 0.836 & 2.333e-02     & 0.790 & 9.148e-02            & 0.453  \\ \hline
				1.017e-02 & 1.100e-03              & 0.889 & 3.913e-03              & 0.848 & 1.342e-02     & 0.792 & 6.596e-02            & 0.468  \\ \hline
				
			\end{tabular}
\\[0.5em]
\begin{tabular}{|c||c|c||c|c||c|c|}
				\hline
				\multicolumn{7}{|c|}{Scheme \eqref{eq:upwind_scheme}}\\
				\hline 
				\hline 
				$h$ & $L^{2}_{\Omega}$ error & EOC & $L^{2}_{\Gamma}$ error & EOC & $V_{h}$ error & EOC\\
				\hline 
				\hline 
				8.638e-01 & 5.093e-02              &   & 1.084e-01              &   & 2.195e-01     &   \\ \hline
				3.664e-01 & 3.020e-02              & 0.609 & 8.669e-02              & 0.260 & 2.056e-01     & 0.760 \\ \hline
				1.685e-01 & 1.156e-02              & 1.236 & 4.137e-02              & 0.952 & 1.319e-01     & 0.571 \\ \hline
				8.343e-02 & 5.254e-03              & 1.122 & 2.077e-02              & 0.980 & 8.042e-02     & 0.703 \\ \hline
				4.077e-02 & 2.504e-03              & 1.035 & 1.045e-02              & 0.959 & 5.019e-02     & 0.658 \\ \hline
				2.044e-02 & 1.209e-03              & 1.055 & 5.121e-03              & 1.033 & 3.015e-02     & 0.738 \\ \hline
				1.022e-02 & 5.952e-04              & 1.022 & 2.529e-03              & 1.018 & 1.849e-02     & 0.705 \\ \hline
			\end{tabular}
			\par\end{centering}
		\caption{EOC for $\Omega$ a cube and $\boldsymbol{V}(\boldsymbol{x})=(11.4301,0, -1)$\label{tab:tangV-advDif}}
	\end{table}
	
	\subsection{Local gravity field modelling}
	
	\label{Loc-grav-field-mod}
	
	In this section we present local gravity field modelling over Slovakia using
	terrestrial gravity data. The goal of this experiment is to compute a disturbing potential using presented FVM schemes with oblique BC from terrestrial measurements 
	and Dirichlet BCs obtained from satellite based model. Then we transform obtained potential to quasi-geoidal heights and compare them with real measurements.
	On the upper and side
	boundaries, the GO\_CONS\_GCF\_2\_DIR\_R5 model \cite{Bruinsma2013} was
	used and on the bottom boundary we used
	the surface gravity disturbances obtained from the available regular
	grid of gravity anomalies, with the resolution $20''\times30''$,
	that was compiled from original gravimetric measurements \cite{Grand2001}.
	The gravity anomalies were transformed into the gravity disturbances
	by official digital vertical reference model DVRM (www.geoportal.sk).
	
	The domain was bounded by $\langle16{}^{\circ},23{}^{\circ}\rangle$
	meridians and $\langle47{}^{\circ},50.5{}^{\circ}\rangle$ parallels.
	The side boundaries were chosen sufficiently far from the area of Slovakia
	in order to mitigate an influence of the prescribed Dirichlet BC generated
	from the satellite-only geopotential model. For more details about
	this influence see \cite{Faskova2010}. The heights were interpolated
	from SRTM30 PLUS model and the upper boundary is in the height
	of 240km above the reference ellipsoid.
	
	Three experiments with the grid density $841\times631\times301$ were performed
	using the FVM schemes \eqref{eq:scheme}, \eqref{eq:upwind_scheme} and \eqref{eq:split_scheme}.
	The accuracy of the simulations was tested using GNSS-leveling. From the available dataset
	of 61 GNSS-leveling benchmarks, three evident outliers were removed.
	Hence, we tested the obtained local quasi-geoid model at 58 points. The results are summarised in Table \ref{TableSR}  and, for the method
	(19), they are visualized in Fig. \ref{fig:potential_slovakia}. We see a comparable precision of all the methods in this experiment. With this grid resolution the standard deviation of residuals between numerical results and measurements for all schemes is around 5cm.

	\begin{table}		
		\centering{}%
		\begin{tabular}{|c|c|c|c|}
			\hline 
			   & method \ref{eq:scheme} & method \ref{eq:upwind_scheme} & method \ref{eq:split_scheme} \\
			\hline 
			Min  & 0.229 & 0.237 & 0.239 \\
			\hline 
			Mean  & 0.326  & 0.330 & 0.331 \\
			\hline 
			Max  & 0.449 & 0.458 & 0.459 \\
			\hline 
			Range  & 0.22 & 0.221 & 0.220 \\
			\hline 
			STD & 0.052  & 0.050  & 0.050 \\
			\hline 
		\end{tabular}
	\caption{The GNSS-leveling test $[m]$ at 58 points in area of Slovakia.\label{TableSR}}
	\end{table}

	\begin{figure}
		\centering
			\includegraphics[width=\textwidth]{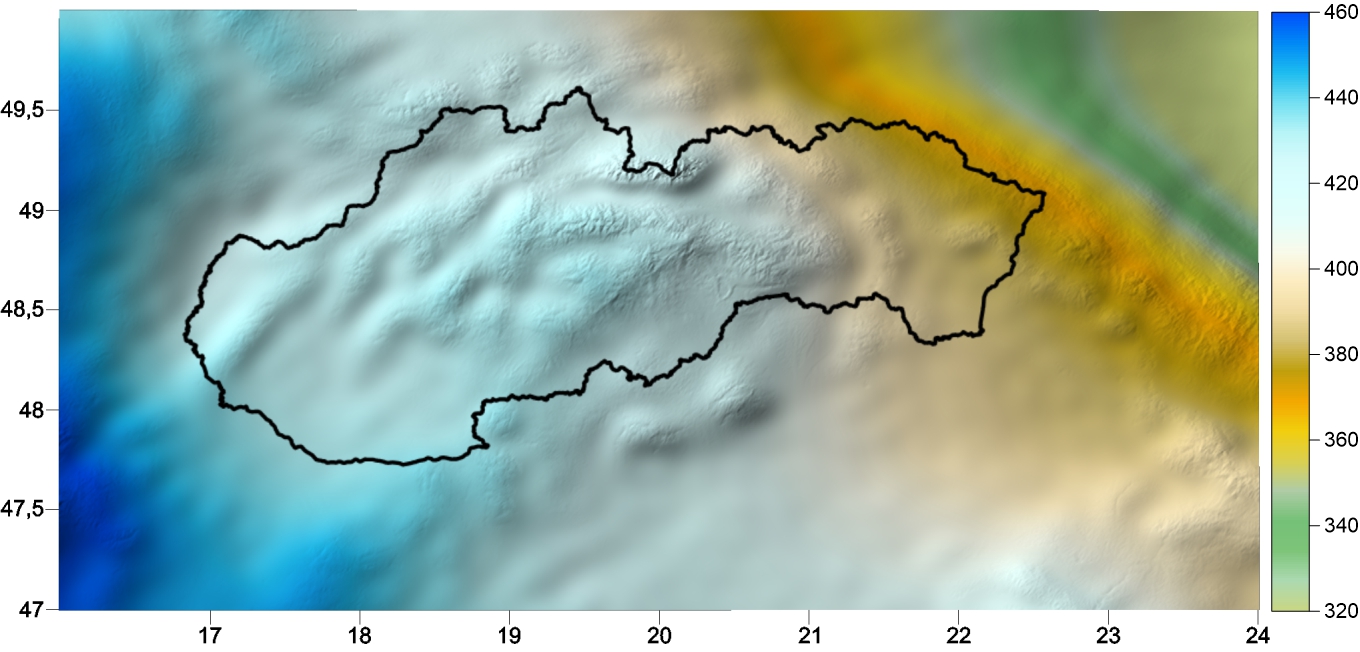}
		\caption{Disturbing potential in the area of Slovakia.}
		\label{fig:potential_slovakia}
	\end{figure}

	\section{Conclusion}\label{sec:conclusion}
	
		We developed a framework for designing and analysing Finite Volume schemes for the Laplace equation with oblique boundary conditions. This framework, which can easily be extended to more general second order differential equations, consists in splitting the boundary condition into a normal and a tangential component, the later being handled as an advection term along the boundary of the domain; to ensure optimal convergence rates, this advection term is discretised using a centered scheme, with added numerical diffusion for stability purposes. The convergence analysis was carried out under usual coercivity and consistency assumptions on the numerical fluxes, and therefore applies to a range of possible FV discretisations. This analysis establishes first-order rates of convergence in a discrete energy ($H^1$) norm.
	
	We then constructed specific fluxes, in the case where the computational domain is discretised using generalised hexahedra, and we identified geometrical conditions, easy to check during simulations, that ensure their coercivity and consistency. Two alternative discretisations of the oblique boundary conditions were also presented: the first one uses an upwind FV discretisation of the boundary advection, the second is not based on a FV discretisation on the boundary, but rather on splitting the outer normal to the boundary into its oblique component, and a tangential component discretised using finite differences and the specific geometry of the mesh.
	
	We then provided extensive numerical tests, designed to assess the accuracy and robustness of the method, for various choices of the computational domain, and of the oblique vector field defining the boundary conditions. These tests confirmed, for all three schemes, the theoretical first-order rate of convergence in energy norm. In some tests, the energy rate of convergence is actually apparently higher than the theoretical one (but the asymptotic convergence rate might not have been attained at the considered mesh sizes). The second variant, based on a splitting of the outer normal, seems to present the best accuracy in our initial tests, when the velocity field is not too tangential to the boundary. For a nearly tangential velocity field, this splitting scheme breaks down as the numerical solver fails to find a solution to it. On the contrary, the other two variants remain robust and convergent in this extreme situation, albeit with a reduced accuracy. All three schemes were used to compute a quasi-geoidal height in the region of Slovakia. For this test, all methods give results with comparable quality.

	\appendix
	
	\section{Proof of Theorem \ref{th:error.estimate}}\label{appen:proof}
	
	The proof hinges on the 3rd Strang lemma of \cite{DPD18}. Let us first recast the scheme \eqref{eq:scheme} under
	a variational formulation. Take $\varphi\in V_h$, multiply \eqref{eq:scheme:balance} by $\varphi_{p}$ and summing over
	$p\in\mesh$ to get
	\begin{align*}
		& \sum_{p\in\mesh}\Bigg(\sum_{\sigma\in \F(p)\backslash \FG}\flux{p,\sigma}(T)\varphi_{p}+\sum_{\sigma\in \F(p)\cap\FG}\sum_{e\in \E(\sigma)}T_{e}\left[\vec{W}\cdot\vec{n}\right]_{\sigma,e}\varphi_{p}\\
		& \qquad-\sum_{\sigma\in \F(p)\cap\FG}\left[\DIVG\vec{W}\right]_{\sigma}T_{p}\varphi_{p}+Rh_\Gamma\sum_{\sigma\in \F(p)\cap\FG}\sum_{e\in \E(\sigma)}\flux[\Gamma]{\sigma,e}(T)\varphi_{p}\Bigg)
		=\sum_{p\in\mesh}\sum_{\sigma\in \F(p)\cap\FG}\iint_\sigma g \varphi_p.
	\end{align*}
	Using the conservativity of the fluxes $\flux{p,\sigma}$ (see \eqref{eq:cons.inner}), $\flux[\Gamma]{\sigma,e}(T)$ (see \eqref{eq:scheme:conserv}) and $T_{e}\brak{\vec{W}\cdot\vec{n}}{\sigma,e}$ (by definition of $\brak{\vec{W}\cdot\vec{n}}{\sigma,e}$), and the zero value of $T_e$ if $e$ is a boundary edge in $\Gamma$, we gather the sums in the left-hand side by faces and edges as in \cite[Proofs of Theorem 27 and 33]{DPD18} to find
	\begin{equation}\label{scheme:recast}
		\begin{aligned}
			& \sum_{\sigma\in \F\backslash\FG}\flux{p,\sigma}(T)\left(\varphi_{p}-\varphi_{q}\right)+\sum_{\sigma\in \FG}\sum_{e\in \E(\sigma)}T_{e}\left[\vec{W}\cdot\vec{n}\right]_{\sigma,e}\left(\varphi_{p}-\varphi_{e}\right)\\
			& \qquad-\sum_{\sigma\in \FG}\left[\DIVG\vec{W}\right]_{\sigma}T_{p}\varphi_{p}+Rh_\Gamma\sum_{\sigma\in \FG}\sum_{e\in \E(\sigma)}\flux[\Gamma]{\sigma,e}(T)\left(\varphi_{p}-\varphi_{e}\right)
			=\sum_{p\in\mesh}\sum_{\sigma\in \F(p)\cap\FG}\iint_\sigma g \varphi_p.
		\end{aligned}
	\end{equation}
	The solution $T\in V_h$ to the scheme thus satisfies $a_h(T,\varphi)=\ell_h(\varphi)$ for all $\varphi\in V_h$,
	with $a_h(T,\varphi)$ (resp. $\ell_h$) the bilinear form (resp. linear form) in the left-hand side (resp. right-hand side) of \eqref{scheme:recast}. Owing to the 3rd Strang lemma \cite[Theorem 10]{DPD18}, the estimate \eqref{eq:error.estimate} follows if we establish the coercivity and consistency properties:
	\begin{equation}\label{prop:coer}
		a_h(\varphi,\varphi)\gtrsim \norm[V_h]{\varphi}^2\quad\forall \varphi\in V_h
	\end{equation}
	and, letting $\mathcal E_h(\overline{T};\varphi):=\ell_h(\varphi)-a_h(\overline{T},\varphi)$ be the consistency error,
	\begin{equation}\label{prop:cons}
		\sup_{\varphi\in V_h,\,\norm[V_h]{\varphi}\le 1}\mathcal E_h(\overline{T};\varphi)\lesssim h\norm[C^2(\overline{\Omega})]{\overline{T}}.
	\end{equation}

	\subsection{Coercivity}
	
	The coercivity properties \eqref{eq:coer.flux.Omega} and \eqref{eq:coer.flux.Gamma} show that
	\begin{equation}\label{coer.1}
		\begin{aligned}
			a_{h}(\varphi,\varphi) \ge{}& \rho_\Omega \seminorm[V_h,\Omega]{\varphi}^{2}+Rh_\Gamma\rho_\Gamma \seminorm[V_h,\Gamma]{\varphi}^{2}\\
			&+\underbrace{\sum_{\sigma\in \FG}\sum_{e\in \E(\sigma)}\varphi_{e}\left[\vec{W}\cdot\vec{n}\right]_{\sigma,e}\left(\varphi_{p}-\varphi_{e}\right)
				-\sum_{\sigma\in \FG}\left[\DIVG\vec{W}\right]_{\sigma}\varphi_{p}^{2}}_{\term_1}.
		\end{aligned}
	\end{equation}
	Simple algebraic identities show that
	\begin{align*}
		\term_1={}& \sum_{\sigma\in \FG}\sum_{e\in \E(\sigma)}\left[\vec{W}\cdot\vec{n}\right]_{\sigma,e}\left(\frac{\varphi_{e}-\varphi_{p}}{2}+\frac{\varphi_{e}+\varphi_{p}}{2}\right)\left(\varphi_{p}-\varphi_{e}\right)-\sum_{\sigma\in \FG}\left[\DIVG\vec{W}\right]_{\sigma}\varphi_{p}^{2}\\
		={}&\sum_{\sigma\in \FG}\sum_{e\in \E(\sigma)}\brak{\vec{W}\cdot\vec{n}}{\sigma,e}\frac{1}{2}\left(-\left(\varphi_{e}-\varphi_{p}\right)^{2}+\varphi_{p}^{2}-\varphi_{e}^{2}\right)-\sum_{\sigma\in \FG}\brak{\DIVG\vec{W}}{\sigma}\varphi_{p}^{2}
	\end{align*}
	By conservativity of $\brak{\vec{W}\cdot\vec{n}}{\sigma,e}$ and zero value of $\varphi_e$ on boundary edges of $\Gamma$,
	\[
	\sum_{\sigma\in \FG}\sum_{e\in \E(\sigma)}\brak{\vec{W}\cdot\vec{n}}{\sigma,e}\varphi_{e}^{2}=
	\sum_{e\in \EGi}\left(\brak{\vec{W}\cdot\vec{n}}{\sigma,e}+\brak{\vec{W}\cdot\vec{n}}{\sigma',e}\right)\varphi_e^2=0.
	\]
	Hence, since $\sum_{e\in \E(\sigma)}\brak{\vec{W}\cdot\vec{n}}{\sigma,e}=\brak{\DIVG\vec{W}}{\sigma}$,
	\begin{equation}\label{coer.2}
		\term_1=-\sum_{\sigma\in \FG}\sum_{e\in \E(\sigma)}\frac{1}{2}\brak{\vec{W}\cdot\vec{n}}{\sigma,e}\left(\varphi_{p}-\varphi_{e}\right)^{2}-\frac{1}{2}\sum_{\sigma\in \FG}\brak{\DIVG\vec{W}}{\sigma}\varphi_{p}^{2}.
	\end{equation}
	Using $\brak{\DIVG\vec{W}}{\sigma}\le \norm[C(\Gamma)]{(\DIVG \vec{W})^+}|\sigma|$ and the trace inequality \eqref{trace:ineq}, we write
	\[
	-\sum_{\sigma\in \FG}\brak{\DIVG\vec{W}}{\sigma}\varphi_{p}^{2}\ge
	-\norm[C(\Gamma)]{(\DIVG\vec{W})^+}\sum_{\sigma\in\FG}|\sigma| \varphi_p^2\ge -\norm[C(\Gamma)]{(\DIVG\vec{W})^+}
	C_{\rm tr}\seminorm[V_h,\Omega]{\varphi}^2.
	\]
	Plugging this into \eqref{coer.2} and noticing, since $d_{pe}^\bot\le h_\Gamma$, that
	\[
	\brak{\vec{W}\cdot\vec{n}}{\sigma,e}\le \norm[C(\Gamma)^d]{\vec{W}}|e|\le h_\Gamma \norm[C(\Gamma)^d]{\vec{W}}\frac{|e|}{d_{pe}^\bot},
	\]
	we obtain
	\[
	\term_1\ge -\frac12 h_\Gamma \norm[C(\Gamma)^d]{\vec{W}}\seminorm[V_h,\Gamma]{\varphi_h}^2-\frac12 C_{\rm tr}\norm[C(\Gamma)]{(\DIVG\vec{W})^+}\seminorm[V_h,\Omega]{\varphi_h}^2.
	\]
	Coming back to \eqref{coer.1}, we infer that
	\begin{equation}\label{coer.3}
		a_{h}(\varphi,\varphi) \ge \left(\rho_\Omega-\frac12 C_{\rm tr}\norm[C(\Gamma)]{(\DIVG\vec{W})^+}\right) \seminorm[V_h,\Omega]{\varphi}^{2}+h_\Gamma\left(R\rho_\Gamma - \frac12 \norm[C(\Gamma)^d]{\vec{W}}\right)\seminorm[V_h,\Gamma]{\varphi}^{2}.
	\end{equation}
	Owing to Assumption \ref{assum:coer}, this proves \eqref{prop:coer}.

	\subsection{Consistency}
	
	Using \eqref{eq:exact_scheme} and recalling that $\ell_h(\varphi)$ is defined by the right-hand side of \eqref{scheme:recast}, we write
	\begin{align}
		\ell_{h}(\varphi) ={}&\sum_{p\in\mesh}\left(\sum_{\sigma\in \F(p)\backslash\FG}\overline{F}_{p,\sigma}(\overline{T})+
		\sum_{\sigma\in \F(p)\cap\FG}\sum_{e\in E(\sigma)}\brak{\,\overline{T}\vec{W}\cdot\vec{n}}{\sigma,e}-\sum_{\sigma\in \F(p)\cap\FG}\brak{\,\overline{T}\DIVG\vec{W}}{\sigma}\right)\varphi_{p}\nonumber\\
		={}&\sum_{\sigma\in \F\backslash\FG}\overline{F}_{p,\sigma}(\overline{T})\left(\varphi_{p}-\varphi_{q}\right)+
		\sum_{\sigma\in \FG}\sum_{e\in \E(\sigma)}\brak{\,\overline{T}\vec{W}\cdot\vec{n}}{\sigma,e}\left(\varphi_{p}-\varphi_{e}\right)-\sum_{\sigma\in \FG}\brak{\,\overline{T}\DIVG\vec{W}}{\sigma}\varphi_{p},\label{eq:discrete_functional_exact}
	\end{align}
	where we have used the conservativity of the fluxes to gather the sums by faces in the second equality. Subtracting $a_h(I_h\overline{T},\varphi)$ (given by the left-hand side of \eqref{scheme:recast} with $T$ replaced by $I_h\overline{T}$), we can split the consistency error into four terms:
	\begin{equation}\label{cons.error.split}
		\mathcal E_h(I_h \overline{T};\varphi)=\term_{c,1}+\term_{c,2}+\term_{c,3}+\term_{c,4}
	\end{equation}
	with, setting $I_h \overline{T}=((\overline{T}_p)_{p\in\mesh},(\overline{T}_\sigma)_{\sigma\in\FD},(\overline{T}_e)_{e\in\EG})$,
	\begin{align*}
		\term_{c,1}={}& \sum_{\sigma\in \F\backslash\FG}\left(\overline{F}_{p,\sigma}(\overline{T})-\flux{p,\sigma}(I_h \overline{T})\right)\left(\varphi_{p}-\varphi_{q}\right),\\
		\term_{c,2}={}& \sum_{\sigma\in \FG}\sum_{e\in \E(\sigma)}\left(\brak{\,\overline{T}\vec{W}\cdot\vec{n}}{\sigma,e}-\overline{T}_{e}\brak{\vec{W}\cdot\vec{n}}{\sigma,e}\right)\left(\varphi_{p}-\varphi_{e}\right),\\
		\term_{c,3}={}& -\sum_{\sigma\in \FG}\left(\brak{\,\overline{T}\DIVG\vec{W}}{\sigma}-\overline{T}_{p}\brak{\DIVG\vec{W}}{\sigma}\right)\varphi_{p},\\
		\term_{c,4}={}& Rh_\Gamma\sum_{\sigma\in \FG}\sum_{e\in \E(\sigma)}\flux[\Gamma]{\sigma,e}(I_h \overline{T})\left(\varphi_{p}-\varphi_{e}\right).
	\end{align*}
	
	We now estimate each of these terms.
	
	\paragraph{Term $\term_{c,1}$.} Introducing $\sqrt{d_{pq}/|\sigma|}$ and using a Cauchy--Schwarz inequality, the consistency property \eqref{eq:cons.flux.Omega} yields
	\begin{align*}
		|\term_{c,1}| \le{}&\left(\sum_{\sigma\in \F\backslash\FG}\frac{d_{pq}}{|\sigma|}\left(\overline{F}_{p,\sigma}(\overline{T})-\flux{p,\sigma}(I_h \overline{T})\right)^{2}\right)^{\frac12}\left(\sum_{\sigma\in \F\backslash\FG}\frac{|\sigma|}{d_{pq}}\left(\varphi_{p}-\varphi_{q}\right)^{2}\right)^{\frac12}\\
		\lesssim{}& h\norm[C^2(\overline{\Omega})]{\overline{T}}\left(\sum_{\sigma\in \F\backslash\FG}d_{pq}|\sigma|\right)^{\frac12}\seminorm[V_h,\Omega]{\varphi}.
	\end{align*}
	Let $D_{p\sigma}$ be the convex hull of $\vec{x}_p$ and $\sigma$. If $\sigma$ is flat, by \cite[Lemma B.2]{gdm} we have $|D_{p\sigma}|=d^\bot_{p,\sigma}|\sigma|/3$ and thus, by definition of ${\rm reg}_\mesh$ (which implies $d_{pq}\le {\rm diam}(p)+{\rm diam}(q)\lesssim {\rm diam}(p)\lesssim d_{p\sigma}^\bot$),
	\begin{equation}\label{est:reg.fac}
		\sum_{\sigma\in \F\backslash\FG}d_{pq}|\sigma|\le \sum_{p\in\mesh}\sum_{\sigma\in\F(p)\backslash\FG}d_{pq}|\sigma|
		\lesssim \sum_{p\in\mesh}\sum_{\sigma\in\F(p)\backslash\FG}|D_{p\sigma}|=\sum_{p\in\mesh}|p|=|\Omega|.
	\end{equation}
	This final estimate also holds in case of non-flat $\sigma$, as can be seen approximating $\sigma$ by piecewise flat surfaces.
	Hence,
	\begin{equation}\label{T1c.h}
		|\term_{c,1}|\lesssim h\norm[C^2(\overline{\Omega})]{\overline{T}}\norm[V_h]{\varphi}.
	\end{equation}

	\paragraph{Term $\term_{c,2}$.} We first estimate the consistency of the fluxes involved in this term.
	Using the definition of the interpolant \eqref{eq:interpolant} we have $\int_e (\overline{T}-\overline{T}_e)=0$ and thus
	\begin{multline}
		\brak{\,\overline{T}\vec{W}\cdot\vec{n}}{\sigma,e}-\overline{T}_{e}\brak{\vec{W}\cdot\vec{n}}{\sigma,e}
		= \int_{e}\overline{T}\vec{W}\cdot\vec{n}_{\sigma,e}-\overline{T}_{e}\int_{e}\vec{W}\cdot\vec{n}_{\sigma,e}
		= \int_{e}\left(\overline{T}-\overline{T}_{e}\right)\vec{W}\cdot\vec{n}_{\sigma,e}\\
		= \int_{e}\left(\overline{T}-\overline{T}_{e}\right)\left(\vec{W}\cdot\vec{n}_{\sigma,e}-\brak{\vec{W}\cdot\vec{n}}{\sigma,e}\right)
		\leq h_\Gamma^{2}|e|\norm[C^2(\overline{\Omega})]{\overline{T}}\norm[C^1(\Gamma)^d]{\vec{W}},
		\label{est.flux.T2c}\end{multline}
	where the conclusion follows from a mean value theorem on $\overline{T}$ and $\vec{W}$.
	Applying a Cauchy--Schwarz inequality and using \eqref{est.flux.T2c}, we infer
	\begin{align}
		|\term_{c,2}| \lesssim{}& \left(\sum_{\sigma\in \FG}\sum_{e\in \E(\sigma)}\frac{d_{pe}^\bot}{|e|}\left(\brak{\,\overline{T}\vec{W}\cdot\vec{n}}{\sigma,e}-\overline{T}_{e}\brak{\vec{W}\cdot\vec{n}}{\sigma,e}\right)^{2}\right)^{\frac12}\left(\sum_{\sigma\in \FG}\sum_{e\in \E(\sigma)}\frac{|e|}{d_{pe}^\bot}\left(\varphi_{p}-\varphi_{e}\right)^{2}\right)^{\frac12}\nonumber\\
		\lesssim{}& h_\Gamma^2 \norm[C^1(\overline{\Omega})]{\overline{T}} \left(\sum_{\sigma\in \FG}\sum_{e\in \E(\sigma)}d_{pe}^\bot |e|\right)^{\frac12}\seminorm[V_h,\Gamma]{\varphi}.
		\label{for:rem.centred}
	\end{align}
	In a similar way as in the last equalities in \eqref{est:reg.fac}, $D_{pe}$ being the convex hull of $\vec{x}_p$ and $e$ we have
	\begin{equation}\label{diam.Gamma}
		\sum_{\sigma\in \FG}\sum_{e\in \E(\sigma)}d_{pe}^\bot|e|
		=2\sum_{\sigma\in \FG}\sum_{e\in \E(\sigma)}|D_{pe}|= 2|\Gamma|.
	\end{equation}
	Since $\seminorm[V_h,\Gamma]{\varphi}\le h_\Gamma^{-\frac12}\norm[V_h]{\varphi}$, we conclude that
	\begin{equation}\label{T2c.h}
		|\term_{c,2}|\lesssim h_\Gamma^{\frac32} \norm[C^2(\overline{\Omega})]{\overline{T}} \norm[V_h]{\varphi}.
	\end{equation}
	
	\begin{remark}[Centred discretisation of the advective term]\label{rem:centred}
		The approximation in the first term of \eqref{eq:approx.brak} corresponds to a centred discretisation of the advection term $\DIVG(\overline{T}\vec{W})$ on $\Gamma$. To stabilise this centred discretisation and ensure the coercivity of the scheme, we have to add the artificial diffusion through the terms $Rh_\Gamma\FlG(T)$ in \eqref{eq:scheme:balance}. A standard option to avoid adding numerical diffusion is to directly use an upwind discretisation of the advective term, as in \eqref{eq:upwind_scheme}. In this case, since stability would not require to introduce artificial diffusion, we would only consider cell unknowns in $V_h$ (and not introduce edge unknowns), and we would take $\norm[V_h]{\cdot}=\seminorm[V_h,\Omega]{\cdot}$.
		The resulting scheme would be \eqref{eq:upwind_scheme}.
		Such a choice, however, would prevent us from introducing $\brak{\vec{W}\cdot\vec{n}}{\sigma,e}$ in \eqref{est.flux.T2c} and the resulting estimate would be in $\mathcal O(h_\Gamma)$ instead of $\mathcal O(h_\Gamma^2)$. Carrying on as in \eqref{for:rem.centred} but with $\norm[V_h,\Omega]{\cdot}$ instead of the absent $\seminorm[V_h,\Gamma]{\cdot}$, with the natural assumption that $|e|h_\Gamma\lesssim |\sigma|$, we would arrive at \eqref{T2c.h} with $h_\Gamma^{\frac12}$ instead of $h_\Gamma^{\frac32}$. The final consistency estimate, and thus error estimate, would then be in $\mathcal O(h^{\frac12})$ instead of $\mathcal O(h)$.
	\end{remark}

	\paragraph{Term $\term_{c,3}$.} Notice first that
	\[
	\left|\brak{\,\overline{T}\DIVG\vec{W}}{\sigma}-\overline{T}_{p}\brak{\DIVG\vec{W}}{\sigma}\right|
	= \left|\iint_{\sigma}\left(\overline{T}-\overline{T}_{p}\right)\DIVG\vec{W}\right|
	\le h_\Gamma \norm[C^1(\overline{\Omega})]{\overline{T}} |\sigma| \norm[C^0(\Gamma)]{\DIVG\vec{W}}.
	\]
	Hence, using a Cauchy--Schwarz inequality and the trace inequality \eqref{trace:ineq},
	\begin{align}
		|\term_{c,3}|\le{}&\left(\sum_{\sigma\in \FG}\frac{1}{|\sigma|}\left(\brak{\,\overline{T}\DIVG\vec{W}}{\sigma}-\overline{T}_{p}\brak{\DIVG\vec{W}}{\sigma}\right)^{2}\right)^{\frac12}\left(\sum_{\sigma\in \FG}|\sigma|\varphi_{p}^{2}\right)^{\frac12}\nonumber\\
		\lesssim{}& h_\Gamma \norm[C^1(\overline{\Omega})]{\overline{T}} \left(\sum_{\sigma\in \FG}|\sigma|\right)^{\frac12}
		C_{\rm tr}\seminorm[V_h,\Omega]{\varphi}\le h_\Gamma \norm[C^2(\overline{\Omega})]{\overline{T}} |\Gamma|^{\frac12}
		C_{\rm tr}\norm[V_h]{\varphi}.
		\label{T3c.h}
	\end{align}
	
	\paragraph{Term $\term_{c,4}$.} Introducing the exact surface fluxes $\overline{F}_{\sigma,e}(\overline{T})=-\int_e \nabla \overline{T}\cdot\vec{n}_{\sigma,e}$ on $\Gamma$, we write
	\[
	\term_{c,4} = Rh_\Gamma\underbrace{\sum_{\sigma\in \FG}\sum_{e\in \E(\sigma)}\left(\flux[\Gamma]{\sigma,e}(I_h \overline{T})-\overline{F}_{\sigma,e}(\overline{T})\right)\left(\varphi_{p}-\varphi_{e}\right)}_{\term_{c,4,1}}+Rh_\Gamma\underbrace{\sum_{\sigma\in \FG}\sum_{e\in \E(\sigma)}\overline{F}_{\sigma,e}(\overline{T})\left(\varphi_{p}-\varphi_{e}\right)}_{\term_{c,4,2}}.
	\]
	A Cauchy--Schwarz inequality, the consistency property \eqref{eq:cons.flux.Gamma}, and \eqref{diam.Gamma} show that 
	\begin{align}
		|\term_{c,4,1}|\le{}& \left(\sum_{\sigma\in \FG}\sum_{e\in \E(\sigma)}\frac{d_{pe}^\bot}{|e|}\left(\flux[\Gamma]{\sigma,e}(I_h \overline{T})-\overline{F}_{\sigma,e}(\overline{T})\right)^2\right)^{\frac12}\left(\sum_{\sigma\in \FG}\sum_{e\in \E(\sigma)}\frac{|e|}{d_{pe}^\bot}\left(\varphi_{p}-\varphi_{e}\right)^{2}\right)^{\frac12}\nonumber\\
		\lesssim{}& h_\Gamma\norm[C^2(\overline{\Omega})]{\overline{T}}\seminorm[V_h,\Gamma]{\varphi}.
		\label{T4c.1}
	\end{align}
	For $\term_{c,4,2}$, we use the conservativity of the fluxes $\overline{F}_{\sigma,e}(\overline{T})$, the fact that $\varphi_e=0$ for edges on the boundary of $\Gamma$, $\sum_{e\in \E(\sigma)}\overline{F}_{\sigma,e}(\overline{T})=-\iint_\sigma \Delta_\Gamma \overline{T}$, and the trace inequality \eqref{trace:ineq} to write
	\begin{align*}
		|\term_{c,4,2}| =\left|\sum_{\sigma\in \FG}\sum_{e\in \E(\sigma)}\overline{F}_{\sigma,e}(\overline{T})\varphi_{p}\right|
		=\left|\sum_{\sigma\in \FG}-\iint_\sigma\Delta_\Gamma \overline{T}\varphi_{p}\right|
		\le{}&\left(\sum_{\sigma\in \FG}\frac{1}{|\sigma|}\left(\iint_\sigma \Delta_\Gamma T\right)^{2}\right)^{\frac12}\left(\sum_{\sigma\in \FG}|\sigma|\varphi_{p}^{2}\right)^{\frac12}\\
		\le{}& \norm[C^2(\overline{\Omega})]{\overline{T}}|\Gamma|^{\frac12}C_{\rm tr}\seminorm[V_h,\Omega]{\varphi}
		\lesssim \norm[C^2(\overline{\Omega})]{\overline{T}}\norm[V_h]{\varphi}.
	\end{align*}
	Combined with \eqref{T4c.1} and recalling that $\term_{4,c}=Rh_\Gamma \term_{4,c,1}+Rh_\Gamma \term_{4,c,2}$ this shows that
	\begin{equation}\label{T4c.h}
		|\term_{c,4}|\lesssim h_\Gamma \norm[C^2(\overline{\Omega})]{\overline{T}}\norm[V_h]{\varphi}.
	\end{equation}
	
	Gathering \eqref{T1c.h}, \eqref{T2c.h}, \eqref{T3c.h} and \eqref{T4c.h} in \eqref{cons.error.split}, we infer that \eqref{prop:cons} holds, which concludes the proof of Theorem \ref{th:error.estimate}.
	
	\section{Proof of Proposition \ref{prop.fluxes}}\label{appen:prop.method}
	
	\subsection{Boundary fluxes}
	
	The coercivity of the HMM fluxes result from the construction of the method as a Gradient Discretisation Method, see \cite[Chapter 13]{gdm} -- we note that this coercivity is purely algebraic, and not impacted by the curvature of the faces $\sigma\in\FG$ or of their edges. In the case of flat faces and edges, the consistency of the fluxes \eqref{hmm:def.flux} is a consequence of \eqref{hmm:cons.grad}, see \cite[Chapter 13]{gdm} or \cite[Example 31]{DPD18}. 
	
	Consider now a curved face $\sigma\in\FG$, and assume that all edges of $\sigma$ are ``only curved along $\sigma$'' in the sense that $\vec{n}_{\sigma,e}$ is constant over $e$ for all $e\in\E(\sigma)$ (see Remark \ref{rem:curved.e} otherwise).
	Because $\Gamma$ is a smooth surface, $\overline{\vec{x}}_e-\vec{x}_p$ is asymptotically close to the tangent plane to $\Gamma$ at any point of $\sigma$ (that is, the projection of $\overline{\vec{x}}_e-\vec{x}_p$ in any normal direction to $\sigma$ has length $\mathcal O(h_\Gamma^2)$). Taylor expansions at any point of $\sigma$ and the consistency property \eqref{hmm:cons.grad} thus give a constant $C$ independent of the mesh such that, for $\varphi\in C^2(\Gamma)$,
	\begin{equation}\label{hmm:cons.T}
		|S_{p,e}(I_h\varphi)|\le C\norm[C^2(\Gamma)]{\varphi}h_\Gamma^2.
	\end{equation}
	Using this estimate and \eqref{hmm:cons.grad}, the arguments developed in \cite[Chapter 13]{gdm} can then easily be adapted and yield \eqref{eq:cons.flux.Gamma}.
	
	\subsection{Inner fluxes}
	
	\subsubsection{Coercivity}
	Without any loss of generality we can select the vertex labels $\vec{x}_{pq}^{\oplus}$ and $\vec{x}_{pq}^{\ominus}$
	(resp. $\vec{x}_{pq}^{\boxplus}$ and $\vec{x}_{pq}^{\boxminus}$) such that $\alpha_{pq}^{\ocircle}\geq0$ (resp. $\alpha_{pq}^{\square}\geq0$). Recalling the definition \eqref{eq:flux_approx} of the fluxes, we use the zero value on the Dirichlet boundary and the Young inequality $xy\ge -\frac12 x^2-\frac12 y^2$ to write
		\begin{align}
		\sum_{\sigma\in \F\backslash\FG}\flux{p,\sigma}{}&(\varphi)\left(\varphi_{p}-\varphi_{q}\right)
		= \sum_{\sigma_{pq}\in \F\backslash\FG}|\sigma_{pq}|\left(\frac{1}{\beta_{pq}}\frac{\varphi_{p}-\varphi_{q}}{d_{pq}}+\frac{\alpha_{pq}^{\ocircle}}{\beta_{pq}}\frac{\varphi_{pq}^{\oplus}-\varphi_{pq}^{\ominus}}{d_{pq}^{\ocircle}}+\frac{\alpha_{pq}^{\square}}{\beta_{pq}}\frac{\varphi_{pq}^{\boxplus}-\varphi_{pq}^{\boxminus}}{d_{pq}^{\square}}\right)\left(\varphi_{p}-\varphi_{q}\right)\nonumber\\
		={}& \sum_{\sigma_{pq}\in \F\backslash(\FG\cup\FD)}|\sigma_{pq}|\left(\frac{1}{\beta_{pq}}\frac{\varphi_{p}-\varphi_{q}}{d_{pq}}+\frac{\alpha_{pq}^{\ocircle}}{\beta_{pq}}\frac{\varphi_{pq}^{\oplus}-\varphi_{pq}^{\ominus}}{d_{pq}^{\ocircle}}+\frac{\alpha_{pq}^{\square}}{\beta_{pq}}\frac{\varphi_{pq}^{\boxplus}-\varphi_{pq}^{\boxminus}}{d_{pq}^{\square}}\right)\left(\varphi_{p}-\varphi_{q}\right)\nonumber\\
		&+\sum_{\sigma_{pq}\in \FD}\frac{|\sigma_{pq}|}{d_{pq}}\frac{1}{\beta_{pq}}(\varphi_{p}-\varphi_{q})^2\nonumber\\
		\geq & \sum_{\sigma_{pq}\in \F\backslash(\FG\cup\FD)}|\sigma_{pq}|\Bigg(\frac{1}{\beta_{pq}d_{pq}}\left(\varphi_{p}-\varphi_{q}\right)^{2}-\frac{\alpha_{pq}^{\ocircle}}{2\beta_{pq}d_{pq}^{\ocircle}}\left(\varphi_{p}-\varphi_{q}\right)^{2}-\frac{\alpha_{pq}^{\square}}{2\beta_{pq}d_{pq}^{\square}}\left(\varphi_{p}-\varphi_{q}\right)^{2}\nonumber\\
		& \qquad\qquad\qquad\qquad\qquad\qquad\qquad-\frac{\alpha_{pq}^{\ocircle}}{2\beta_{pq}d_{pq}^{\ocircle}}\left(\varphi_{pq}^{\oplus}-\varphi_{pq}^{\ominus}\right)^{2}-\frac{\alpha_{pq}^{\square}}{2\beta_{pq}d_{pq}^{\square}}(\varphi_{pq}^{\boxplus}-\varphi_{pq}^{\boxminus})^{2}\Bigg)\nonumber\\
		&+\sum_{\sigma_{pq}\in \FD}\frac{|\sigma_{pq}|}{d_{pq}}\frac{1}{\beta_{pq}}(\varphi_{p}-\varphi_{q})^2.
		\label{eq:coer.1}
	\end{align}
	
	In order to establish \eqref{eq:coer.flux.Omega}, we now need to find a lower bound of this quantity in terms of sums of $(\varphi_a-\varphi_b)^2$ for $(a,b)$ pairs of neighbouring control volumes.
	The first stage is to recast $\varphi_{pq}^{\oplus}-\varphi_{pq}^{\ominus}$ and $\varphi_{pq}^{\boxplus}-\varphi_{pq}^{\boxminus}$ as combinations of differences of $\varphi$ on neighbouring control volumes. Without loss of generality, we consider $\varphi_{pq}^{\oplus}-\varphi_{pq}^{\ominus}$. We have to deal with three cases, depending if the corresponding vertices are both internal, if one lies on $\Gamma$, or if one lies on the Dirichlet boundary $\partial\Omega\backslash\Gamma$.
	
	\medskip
	
	\textbf{Case 1: Internal vertices}. We assume here that $\vec{x}_{pq}^\oplus$ and $\vec{x}_{pq}^\ominus$ are both in $\Omega$.
	Let $\vec{x}_{pq}^\genexp$ denote any one of these two vertices. Recalling the definitions in Section \ref{sec:prop.fluxes} (see also Figure \ref{fig:neigh}) of $F_{r,pq}^\genexp$ and $e^\genexp_{r,pq}$, we see that the set $\cup_{r=p,q} (F_{r,pq}^\genexp\cup\{r,e_{r,pq}^\genexp\})$ is made of the eight control volumes around $\vec{x}_{pq}^\genexp$ whose unknowns are involved in the definition \eqref{def:varphipq} of $\varphi_{pq}^\genexp$. Hence, we can decompose $\varphi_{pq}^\oplus-\varphi_{pq}^\ominus$ as
	\[
	\varphi_{pq}^{\oplus}-\varphi_{pq}^{\ominus}
	= \frac{1}{8}\sum_{r=p,q}\left(\varphi_{e_{r,pq}^{\oplus}}+\sum_{f\in F_{r,pq}^{\oplus}}\varphi_{f}+\varphi_{r}-\varphi_{r}-\sum_{f\in F_{r,pq}^{\ominus}}\varphi_{f}-\varphi_{e_{r,pq}^{\ominus}}\right).
	\]
	Inside the sum in the right-hand side, each cell unknown appears with the coefficients represented in Figure \ref{fig:coefs} (left).
	Our goal is to gather these terms together in order to write $\varphi_{pq}^{\oplus}-\varphi_{pq}^{\ominus}$ as a combination of terms $\varphi_{b}-\varphi_{a}$ with $a$ and $b$ neighbouring control volumes. This is done by splitting the coefficients in order to associate (parts of) each cell unknown with a neighbouring cell unknown, as in Figure \ref{fig:coefs} (right). 
	
	\begin{figure}[h!]
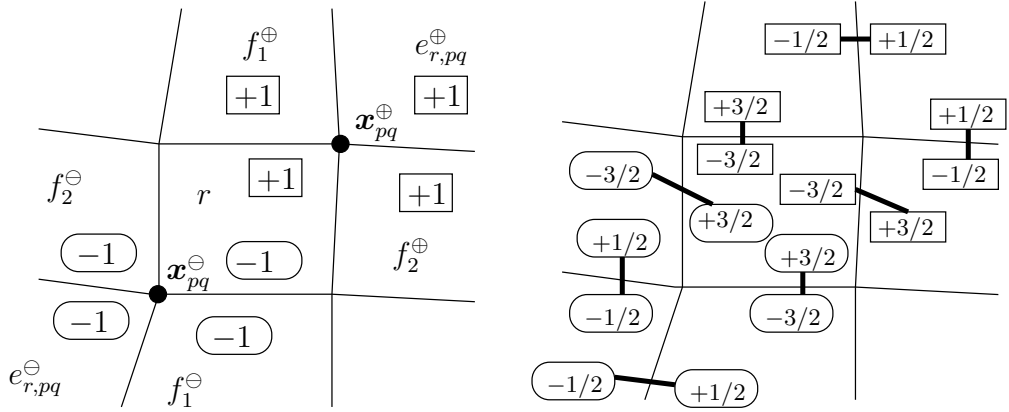

		\centering
		\begin{tabular}{c@{\qquad}c}
			\input{\figsFolfer/fig-coefs.pdf_t}&\input{\figsFolfer/fig-coefs2.pdf_t}
		\end{tabular}
		\caption{Internal vertices. Left: coefficients associated to each cell unknown in the expression of $\varphi_{pq}^{\oplus}-\varphi_{pq}^{\ominus}$ (projection on 2D). Right: splitting of these coefficients, and associations of neighbouring control volumes. 
			The sums of coefficients in each cell are the same in both pictures.
		}\label{fig:coefs}
	\end{figure}
	
	This consists in writing
	\begin{align*}
		\varphi_{pq}^{\oplus}-\varphi_{pq}^{\ominus}={} & \frac{1}{8}\sum_{r=p,q}\Bigg(\frac{\varphi_{e_{r,pq}^{\oplus}}}{2}+\frac{\varphi_{e_{r,pq}^{\oplus}}}{2}+\sum_{f\in F_{r,pq}^{\oplus}}\left(3\frac{\varphi_{f}}{2}-\frac{\varphi_{f}}{2}\right)+3\frac{\varphi_{r}}{2}+3\frac{\varphi_{r}}{2}\\
		& \qquad\qquad-3\frac{\varphi_{r}}{2}-3\frac{\varphi_{r}}{2}-\sum_{f\in F_{r,pq}^{\ominus}}\left(3\frac{\varphi_{f}}{2}-\frac{\varphi_{f}}{2}\right)-\frac{\varphi_{e_{r,pq}^{\ominus}}}{2}-\frac{\varphi_{e_{r,pq}^{\ominus}}}{2}\Bigg)\\
		={}&\frac{1}{16}  \sum_{r=p,q}\left(\sum_{f\in F_{r,pq}^{\oplus}}\left((\varphi_{e_{r,pq}^{\oplus}}-\varphi_{f})+3\left(\varphi_{f}-\varphi_{r}\right)\right)+\sum_{f\in F_{r,pq}^{\ominus}}\left(3\left(\varphi_{r}-\varphi_{f}\right)+(\varphi_{f}-\varphi_{e_{r,pq}^{\ominus}})\right)\right).
	\end{align*}
	We simplify this expression by gathering the terms involving $F_{r,pq}^\oplus$ and $F_{r,pq}^\ominus$ under a sum $\sum_{\genexp\in\{\oplus,\ominus\}}$: setting $\delta_\oplus=+1$ and $\delta_\ominus=-1$, we have
	\[
	\varphi_{pq}^{\oplus}-\varphi_{pq}^{\ominus}=\frac{1}{16}  \sum_{r=p,q}\sum_{\genexp\in\{\oplus,\ominus\}}\sum_{f\in F_{r,pq}^{\genexp}}\delta_\genexp\left((\varphi_{e_{r,pq}^{\genexp}}-\varphi_{f})+3\left(\varphi_{f}-\varphi_{r}\right)\right).
	\]
	Using the definition \eqref{def:eps.zeta}, we arrive at
	\begin{equation}\label{est.phi}
		\varphi_{pq}^{\oplus}-\varphi_{pq}^{\ominus}
		=\frac{1}{16}\sum_{r=p,q}\sum_{\genexp\in\{\oplus,\ominus\}}\sum_{f\in F_{r,pq}^{\genexp}}\delta_\genexp\left(\zeta_{X,pq}^\genexp(\varphi_{e_{r,pq}^{\genexp}}-\varphi_{f})+\zeta_{Y,pq}^\genexp(\varphi_{f}-\varphi_{r})\right).
	\end{equation}
	
	\medskip
	
	\textbf{Case 2: Vertex on $\Gamma$}. One of $\vec{x}_{pq}^\oplus$ or $\vec{x}_{pq}^\ominus$ lies on $\Gamma$. Without loss of generality we assume it is $\vec{x}_{pq}^\ominus$.
	The boundary value $\varphi_{pq}^\ominus$ is expressed as $1/4$ of the sum of the unknowns in four faces lying on $\Gamma$ or, equivalently, $1/8$ of the sum of these four values associated with coefficients $2$. Plugging this expression into $\varphi_{pq}^\oplus-\varphi_{pq}^\ominus$ and reasoning as in Case 1, but using this time the splitting of coefficients represented in Figure \ref{fig:coefs-Gamma}, we arrive at
	\begin{equation}\label{diff.phi}
		\varphi_{pq}^{\oplus}-\varphi_{pq}^{\ominus}=\frac{1}{16}  \sum_{r=p,q}\sum_{f\in F_{r,pq}^{\oplus}}\left((\varphi_{e_{r,pq}^{\oplus}}-\varphi_{f})+3\left(\varphi_{f}-\varphi_{r}\right)\right)
		+\frac{1}{16}\sum_{r=p,q}\sum_{f\in F_{r,pq}^{\ominus}}4(\varphi_r-\varphi_f)
	\end{equation}
	(the sum over $f\in F_{r,pq}^{\ominus}$ actually only contains one term, but is written this way for homogeneity of notations). This sum can be written in the form \eqref{est.phi}, owing to the definition \eqref{def:eps.zeta} of $(\zeta_{X,pq}^\genexp,\zeta_{Y,pq}^\genexp)$.
	
	
	\begin{figure}[h!]
		\centering
		\begin{tabular}{c@{\qquad}c}
			\input{\figsFolfer/fig-coefs-Gamma.pdf_t}&\input{\figsFolfer/fig-coefs-Gamma2.pdf_t}
		\end{tabular}
		\caption{$\vec{x}_{pq}^\ominus$ on $\Gamma$. Left: coefficients associated to each cell unknown in the expression of $\varphi_{pq}^{\oplus}-\varphi_{pq}^{\ominus}$ (projection on 2D). Right: splitting of these coefficients, and associations of neighbouring control volumes.  The sums of coefficients in each cell are the same in both pictures.}\label{fig:coefs-Gamma}
	\end{figure}
	
	\medskip
	
	\textbf{Case 3: Vertex on the Dirichlet boundary}. Assuming $\vec{x}_{pq}^\ominus\in\partial\Omega\backslash\Gamma$, we have $\varphi_{pq}^\ominus=0$.
	Then $F_{r,pq}^\ominus$ is made of the unique face $f^\ominus=\sigma$ (degenerate cell) of $r$ on $\partial\Omega\backslash\Gamma$, associated with a value $\varphi_{f^\ominus}=0$. The splitting of coefficients described in Figure \ref{fig:coefs-dir} leads to \eqref{diff.phi} with the last coefficient 4 repla\-ced by 8; thus, recalling the definition of $(\zeta_{X,pq}^\genexp,\zeta_{Y,pq}^\genexp)$ in \eqref{def:eps.zeta}, we can again write $ \varphi_{pq}^{\oplus}-\varphi_{pq}^{\ominus}$ in the form \eqref{est.phi}.
	
	\begin{figure}[h!]
		\centering
		\begin{tabular}{c@{\qquad}c}
			\input{\figsFolfer/fig-coefs-dir.pdf_t}&\input{\figsFolfer/fig-coefs-dir2.pdf_t}
		\end{tabular}
		\caption{$\vec{x}_{pq}^\ominus$ on $\partial\Omega\backslash\Gamma$. Left: coefficients associated to each cell unknown in the expression of $\varphi_{pq}^{\oplus}-\varphi_{pq}^{\ominus}=\varphi_{pq}^{\oplus}$ (projection on 2D). Right: splitting of these coefficients, and associations of neighbouring control volumes.  The sums of coefficients in each cell are the same in both pictures.
		}\label{fig:coefs-dir}
	\end{figure}
	
	\medskip
	
	\textbf{Conclusion.} We established that the formula \eqref{est.phi} always holds, no matter the positions of the vertices. Accounting for the definitions of $F_{r,pq}^\genexp$ and of $(\zeta_{X,pq}^\genexp,\zeta_{Y,pq}^\genexp)$ (see \eqref{def:eps.zeta}), we see that the right-hand side of this relation is made of at most 32 sums of $\delta_\genexp(\varphi_a-\varphi_b)$ (with $(a,b)=(e_{r,pq}^\genexp,f)$ or $(a,b)=(f,r)$). Hence, using the convexity of the square function, we obtain
\[
		\left(\frac{\varphi_{pq}^{\oplus}-\varphi_{pq}^{\ominus}}{2}\right)^2
		\le\frac{1}{32}\sum_{r=p,q}\sum_{\genexp\in\{\oplus,\ominus\}}\sum_{f\in F_{r,pq}^{\genexp}}\left(\zeta_{X,pq}^\genexp(\varphi_{e_{r,pq}^{\genexp}}-\varphi_{f})^2+\zeta_{Y,pq}^\genexp(\varphi_{f}-\varphi_{r})^2\right).
\]

A similar estimate can be obtained for $(\varphi_{pq}^{\boxplus}-\varphi_{pq}^{\boxminus})^2$, by replacing the sum range $\genexp\in\{\oplus,\ominus\}$ with $\genexp\in\{\boxplus,\boxminus\}$. By plugging these bounds into \eqref{eq:coer.1} we obtain
	\begin{align}
		{}&\sum_{\sigma\in \F\backslash\FG}\flux{p,\sigma}(\varphi)\left(\varphi_{p}-\varphi_{q}\right)\nonumber\\
		&\geq \sum_{\sigma_{pq}\in \F\backslash\FG}|\sigma_{pq}|\left[\frac{1}{\beta_{pq}d_{pq}}-\epsilon_{pq}\frac{\alpha_{pq}^{\ocircle}}{2\beta_{pq}d_{pq}^{\ocircle}}-\epsilon_{pq}\frac{\alpha_{pq}^{\square}}{2\beta_{pq}d_{pq}^{\square}}\right](\varphi_{p}-\varphi_{q})^{2}\nonumber\\
		& \quad-\frac{1}{16}\sum_{\sigma_{pq}\in \F\backslash(\FG\cup\FD)}\sum_{r=p,q}\sum_{\genexp\in\{\oplus,\ominus,\boxplus,\boxminus\}}\frac{|\sigma_{pq}|\alpha_{pq}^{\diamondsuit}}{\beta_{pq}d_{pq}^{\diamondsuit}}\sum_{f\in F_{r,pq}^{\genexp}}\left(\zeta_{X,pq}^\genexp(\varphi_{e_{r,pq}^{\genexp}}-\varphi_{f})^{2}+\zeta_{Y,pq}^{\genexp}(\varphi_{f}-\varphi_{r})^{2}\right),
		\label{eq:coer.2}
	\end{align}
	where $\diamondsuit$ is given by \eqref{def:diamondsuit}, and we have used the definition \eqref{def:eps.zeta} of $\epsilon_{pq}$ to integrate the last term in \eqref{eq:coer.1} into the first one in the right-hand side above.
	
	All the differences of $\varphi$ in this equation are differences $(\varphi_a-\varphi_b)^2$ of values across a face $\sigma_{ab}\in \F\backslash\FG$. For such a given face, the sets $X_{ab}$ and $Y_{ab}$ defined by $\eqref{def:XYab}$ precisely identify the indices in the second addend of \eqref{eq:coer.2} that involve the term $(\varphi_a-\varphi_b)^2$. Hence, \eqref{eq:coer.2} can be re-arranged as
	\begin{align*}
		\sum_{\sigma\in \F\backslash\FG}\flux{p,\sigma}(\varphi)\left(\varphi_{p}-\varphi_{q}\right){}&
		\geq \sum_{\sigma_{ab}\in \F\backslash\FG}\frac{|\sigma_{ab}|}{d_{ab}}\Bigg\{\left[\frac{1}{\beta_{ab}}-\epsilon_{ab}\frac{\alpha_{ab}^{\ocircle}d_{ab}}{2\beta_{ab}d_{ab}^{\ocircle}}-\epsilon_{ab}\frac{\alpha_{ab}^{\square}d_{ab}}{2\beta_{ab}d_{ab}^{\square}}\right]\\
		& -\frac{1}{16}\sum_{(p,q,\genexp)\in X_{ab}}\zeta_{X,pq}^\genexp\frac{|\sigma_{pq}|d_{ab}\alpha_{pq}^{\diamondsuit}}{|\sigma_{ab}|d_{pq}^{\diamondsuit}\beta_{pq}}-\frac{1}{16}\sum_{(p,q,\genexp)\in Y_{ab}}\zeta_{Y,pq}^\genexp\frac{|\sigma_{pq}|d_{ab}\alpha_{pq}^{\diamondsuit}}{|\sigma_{ab}|d_{pq}^{\diamondsuit}\beta_{pq}}\Bigg\}(\varphi_{a}-\varphi_{b})^{2}.
	\end{align*}
	The definition \eqref{def:varrho.h} then shows that \eqref{eq:coer.flux.Omega} holds with $\varrho_{\Omega}=\varrho_{\mesh,\Omega}$.

	\begin{remark}[Alternative coercivity factor]\label{rem:alt.varrho}
		For each $\sigma_{pq}\in\F\backslash\FG$, take $\varpi_{pq}>0$. Using before \eqref{eq:coer.1} the generalised Young inequality $xy\ge -\frac{\varpi_{pq}^{-1}}{2}x^2-\frac{\varpi_{pq}}{2}y^2$, instead of the standard one with $\varpi_{pq}=1$, the reasoning above shows that the regularity factor $\varrho_{\mesh,\Omega}$ can be re-defined such that
		\begin{equation}\label{def:reg.int.3}
			\begin{aligned}
				\varrho_{\mesh,\Omega}:=\max\Bigg\{{}&\left[\frac{1}{\beta_{ab}}-\epsilon_{ab}\frac{\varpi_{ab}^{-1}\alpha_{ab}^{\ocircle}d_{ab}}{2\beta_{ab}d_{ab}^{\ocircle}}-\epsilon_{ab}\frac{\varpi_{ab}^{-1}\alpha_{ab}^{\square}d_{ab}}{2\beta_{ab}d_{ab}^{\square}}\right]-\frac{1}{16}\sum_{(p,q,\genexp)\in X_{ab}}\zeta_{X,pq}^\genexp\frac{\varpi_{pq}|\sigma_{pq}|d_{ab}\alpha_{pq}^{\diamondsuit}}{|\sigma_{ab}|d_{pq}^{\diamondsuit}\beta_{pq}}\\
				& \quad-\frac{1}{16}\sum_{(p,q,\genexp)\in Y_{ab}}\zeta_{Y,pq}^\genexp\frac{\varpi_{pq}|\sigma_{pq}|d_{ab}\alpha_{pq}^{\diamondsuit}}{|\sigma_{ab}|d_{pq}^{\diamondsuit}\beta_{pq}}\,:\,\sigma_{ab}\in\F\backslash\FG\Bigg\}.
			\end{aligned}
		\end{equation}
		For certain choices of $\varpi_{ab}$, this alternative coercivity factor could remain positive and bounded above for certain meshes, for which the factor satisfying \eqref{def:varrho.h} is negative.
	\end{remark}

	\subsubsection{Consistency}\label{sec:cons.in.flux}
	
	We start with a preliminary estimate. Let $\vec{a}=\big(\frac{1}{\beta_{pq}},-\frac{\alpha_{pq}^{\ocircle}}{\beta_{pq}},-\frac{\alpha_{pq}^{\square}}{\beta_{pq}}\big)$. The relation \eqref{npq:rep} yields
	\[
	\vec{a}  =\Big[\vec{s}_{pq},\vec{t}_{pq}^{\ocircle},\vec{t}_{pq}^{\square}\Big]^{-1}\frac{\widetilde{\vec{n}}_{pq}}{|\sigma_{pq}|}.
	\]
	Since $\frac{|\widetilde{\vec{n}}_{pq}|}{|\sigma_{pq}|}\le 1$ (see \eqref{npq:ave}) we infer
	$|\vec{a}|\le \norm{\big[\vec{s}_{pq},\vec{t}_{pq}^{\ocircle},\vec{t}_{pq}^{\square}\big]^{-1}}$.
	The vectors $\vec{s}_{pq},\vec{t}_{pq}^{\ocircle},\vec{t}_{pq}^{\square}$ having unit length, the representation of the inverse of $\big[\vec{s}_{pq},\vec{t}_{pq}^{\ocircle},\vec{t}_{pq}^{\square}\big]$ using the co-matrix and the determinant give a universal constant $C$ such that
	\begin{equation}
		\left|\left(\frac{1}{\beta_{pq}},-\frac{\alpha_{pq}^{\ocircle}}{\beta_{pq}},-\frac{\alpha_{pq}^{\square}}{\beta_{pq}}\right)\right|\leq\frac{C}{\left|\det(\vec{s}_{pq},\vec{t}_{pq}^{\ocircle},\vec{t}_{pq}^{\square})\right|}.\label{eq:coeff_boundness}
	\end{equation}
	
	Let $u\in C^2(\overline{\Omega})$ with $u=0$ on $\partial\Omega\backslash \Gamma$. In the following, we write $\mathcal O(s)$ for generic functions that satisfy $|\mathcal O(s)|\le C\norm[C^2(\overline{\Omega})]{u}|s|$ with $C$ depending only on an upper bound of the regularity factors ${\rm reg}_\mesh$ and ${\rm reg}_{\mesh,\Omega}$ defined by \eqref{def:regh} and \eqref{regfac:cons}. This notation naturally extends to the case where $s$ is a vector.
	
	Taking an arbitrary point $\vec{x}_\sigma\in\sigma_{pq}$, \eqref{npq:ave} and \eqref{npq:rep} show that
	\begin{align}
		\overline{F}_{p,\sigma}(u) ={}&\iint_{\sigma_{pq}}\nabla u\cdot\vec{n}_{pq}
		=\nabla u(\vec{x}_{\sigma})\cdot\iint_{\sigma_{pq}}\vec{n}_{pq}+|\sigma_{pq}|\mathcal O(h)\nonumber\\
		={}&|\sigma_{pq}|\left(\frac{1}{\beta_{pq}}\nabla u(\vec{x}_{\sigma})\cdot\vec{s}_{pq}-\frac{\alpha^{\ocircle}_{pq}}{\beta_{pq}}\nabla u(\vec{x}_{\sigma})\cdot\vec{t}_{pq}^{\ocircle}-\frac{\alpha_{pq}^{\boxempty}}{\beta_{pq}}\nabla u(\vec{x}_{\sigma})\cdot\vec{t}_{pq}^{\boxempty}\right)+|\sigma_{pq}|\mathcal O(h).
		\label{eq:real_flux_approx}
	\end{align}
	Let us look at each directional derivative separately. Since $d_{pq}\le 2h$, the definition \eqref{def:spq} of $\vec{s}_{pq}$ and a Taylor expansion yield
	\begin{equation}
		\nabla u(\vec{x}_{\sigma})\cdot\vec{s}_{pq}=
		\frac{u(\vec{x}_{p})-u(\vec{x}_{q})}{d_{pq}}+\mathcal O(h).
		\label{eq:s_der_approx}
	\end{equation}
	For the derivative in the tangential direction $\vec{t}_{pq}^\ocircle$, Lemma \ref{lem:bary} below shows that
	\begin{align*}
		\nabla u(\vec{x}_{\sigma})\cdot\vec{t}_{pq}^{\ocircle} ={}&\frac{ u\left(\vec{x}_{pq}^{\oplus}\right)-u\left(\vec{x}_{pq}^{\ominus}\right)}{d_{pq}^{\ocircle}}+\mathcal O(h)\\
		={}&\frac{\frac{1}{\Card{\mathcal R(\vec{x}_{pq}^{\oplus})}}\sum_{\vec{y}\in \mathcal R(\vec{x}_{pq}^{\oplus})}u(\vec{y})+\mathcal O(\vec{d}_{pq}^{\oplus 2})-\frac{1}{\Card{\mathcal R(\vec{x}_{pq}^{\ominus})}}\sum_{\vec{y}\in \mathcal R(\vec{x}_{pq}^{\ominus})}u(\vec{y})+\mathcal O(\vec{d}_{pq}^{\ominus 2})}{d_{pq}^{\ocircle}}+\mathcal O(h),
	\end{align*}
	with $\vec{d}_{pq}^{\genexp 2}$ the vector obtained by component-wise squaring $\vec{d}_{pq}^\genexp$.
	Using the definition of ${\rm reg}_{\mesh,\Omega}$ we infer
	\begin{equation}
		\nabla u(\vec{x}_{\sigma})\cdot\vec{t}_{pq}^{\ocircle} =\frac{\frac{1}{\Card{\mathcal R(\vec{x}_{pq}^{\oplus})}}\sum_{\vec{y}\in \mathcal R(\vec{x}_{pq}^{\oplus})}u(\vec{y})-\frac{1}{\Card{\mathcal R(\vec{x}_{pq}^{\ominus})}}\sum_{\vec{y}\in \mathcal R(\vec{x}_{pq}^{\ominus})}u(\vec{y})}{d_{pq}^{\ocircle}}+\mathcal O(h).\label{eq:t_der_approx}
	\end{equation}
	Similarly,
	\begin{equation}
		\nabla u(\vec{x}_{\sigma})\cdot\vec{t}_{pq}^{\boxempty} =\frac{\frac{1}{\Card{\mathcal R(\vec{x}_{pq}^{\boxplus})}}\sum_{\vec{y}\in \mathcal R(\vec{x}_{pq}^{\boxplus})}u(\vec{y})-\frac{1}{\Card{\mathcal R(\vec{x}_{pq}^{\boxminus})}}\sum_{\vec{y}\in \mathcal R(\vec{x}_{pq}^{\boxminus})}u(\vec{y})}{d_{pq}^{\boxempty}}+\mathcal O(h).\label{eq:f_der_approx}
	\end{equation}
	Plug \eqref{eq:s_der_approx}--\eqref{eq:f_der_approx} into \eqref{eq:real_flux_approx} and subtract $\flux{p,\sigma}(I_hu)$ defined by \eqref{eq:flux_approx} with $\varphi=I_hu$, so that $\varphi_p=u(\vec{x}_p)$ for all $p\in\mesh$ and $\varphi_{\vec{y}}=u(\vec{x}_q)=0$ if $q=\sigma\in\FD$. This gives
	\begin{align}
		\overline{F}_{p,\sigma}(u)-\flux{p,\sigma}(I_{h}u) & =|\sigma_{pq}|\left(\frac{\mathcal O(h)}{\beta_{pq}}-\frac{\alpha_{pq}^\ocircle\mathcal O(h)}{\beta_{pq}}-\frac{\alpha_{pq}^\boxempty\mathcal O(h)}{\beta_{pq}}\right)+|\sigma_{pq}|\mathcal O(h).
	\end{align}
	Estimate \eqref{eq:coeff_boundness} and the definition \eqref{regfac:cons} of ${\rm reg}_{\mesh,\Omega}$ then conclude the proof of \eqref{eq:cons.flux.Omega}, with a constant that only depends on an upper bound of this regularity factor.

	\begin{lemma}[Error for smooth functions on barycentric combinations]\label{lem:bary}
		Let $U$ be a convex open set of $\R^3$, $(\vec{z}_i)_{i=1,\ldots,I}$ be points in $U$, and let $\vec{z}=\sum_{i=1}^I \lambda_i \vec{z}_i$ for some convex coefficients $(\lambda_i)_{i=1,\ldots,I}$. If $\psi\in C^2(\overline{U})$ then
		\begin{equation}
			\left|\psi(\vec{z})-\sum_{i=1}^I \lambda_i\psi(\vec{z}_i)\right|\le \frac12 \norm[C^2(\overline{U})]{\psi}\max_{i=1,\ldots,I}|\vec{z}_i-\vec{z}|^2.
			\label{eq:conv.approx}
		\end{equation}
	\end{lemma}
	
	\begin{proof}
		This lemma is classical but its (short) proof is recalled for the sake of legibility. A Taylor expansion around $\vec{z}$ gives 
		\begin{equation}\label{psi.conv}
			\psi(\vec{x})=\psi(\vec{z})+\nabla \psi(\vec{z})\cdot(\vec{x}-\vec{z})+{\rm Rem}(\vec{x},\vec{z}),
		\end{equation}
		where $|{\rm Rem}(\vec{x},\vec{z})|\le \frac12  \norm[C^2(\overline{U})]{\psi}|\vec{x}-\vec{z}|^2$. Apply \eqref{psi.conv} to $\vec{x}=\vec{z_i}$, multiply by $\lambda_i$ and sum over $i=1,\ldots,I$. Since $\sum_{i=1}^I\lambda_i(\vec{z}_i-\vec{z})=0$ the term involving $\nabla\psi(\vec{z})$ disappears and \eqref{eq:conv.approx} follows. \end{proof}
	
\medskip

\thanks{\textbf{Acknowledgement}: this research was supported by the Australian Government through the Australian Research Council's Discovery Projects funding scheme (pro\-ject number DP170100605), and by the Slovak Research and Development Agency (grant APVV-0522-15).
}

\begin{footnotesize}
  \bibliographystyle{abbrv}
  \bibliography{oblique-fv}
\end{footnotesize}

\end{document}